\documentclass[12pt]{amsart}
\usepackage{amsmath, amsthm, amssymb,mathscinet}
\usepackage{graphicx}
\usepackage{mathbbol}
\usepackage{txfonts}
\usepackage[usenames]{color}
\usepackage{hyperref}      

\newtheorem{thm}{Theorem}[section]
\newtheorem{ex}[thm]{Example}
   
\newtheorem{lem}[thm]{Lemma}   
\newtheorem{prop}[thm]{Proposition}
\newtheorem{defn}[thm]{Definition}
\newtheorem{rmrk}[thm]{Remark}  
   
\newcommand{\be}{\begin{equation}}
\newcommand{\ee}{\end{equation}}
\newcommand{\bee}{\begin{equation*}}
\newcommand{\eee}{\end{equation*}}

\newcommand{\N}{\mathbb{N}}

\newcommand{\R}{\mathbb{R}}
\newcommand{\One}{{\bf \rm{1}}}

\newcommand{\Z}{\mathbb{Z}}

\newcommand{\diam}{\operatorname{Diam}}

\newcommand{\vol}{\operatorname{Vol}}
\newcommand{\ric}{\operatorname{Ric}}
\newcommand{\cur}{\operatorname{R}}

\newcommand{\Hto}{\stackrel { \textrm{H}}{\longrightarrow} }
\newcommand{\GHto}{\stackrel { \textrm{GH}}{\longrightarrow} }
\newcommand{\Fto}{\stackrel {\mathcal{F}}{\longrightarrow} }

\newcommand{\Fm}{{\mathcal F}}

\newcommand{\set}{\rm{set}}
\newcommand{\Lip}{\operatorname{Lip}}

\newcommand{\mass}{{\mathbf M}}
\newcommand{\curr}{{\mathbf M}}         
\newcommand{\intrectcurr}{{\mathcal I}} 
\newcommand{\intcurr}{{\mathbf I}}      

\newcommand{\fillvol}{{\operatorname{FillVol}}}

\newcommand{\rstr}{\:\mbox{\rule{0.1ex}{1.2ex}\rule{1.1ex}{0.1ex}}\:}
\newcommand{\bdry}{\partial}

\begin{document}

\title{Convergence of Manifolds and Metric Spaces with Boundary}
\author{Raquel Perales}

\newpage
\thanks{The author gratefully acknowledges financial support from the National Science Foundation under Grant No. 0932078 000 via the Mathematical Science Research Institute, Berkeley, where the author was a visitor for a month in the Fall Semester 2013, and under NSF DMS 10060059 as a funded doctoral student of Prof. Christina Sormani. In addition, the author received a Summer Research Award for doctoral students from the Stony Brook University Department of Mathematics.}
\address{CONACYT Research Fellow at IMATE UNAM, Oaxaca, Mexico}
\email{raquel.peralesaguilar@gmail.com}

\keywords{}


\begin{abstract}
  We study sequences of oriented Riemannian manifolds with boundary
and, more generally, integral current spaces and metric spaces 
with boundary.  {\color{blue}For a metric space, we define its boundary to be the completion of the space minus the space. We impose conditions on these spaces in order to get  Gromov-Hausdorff (GH) subconvergence. By adding some other conditions}  we prove theorems demonstrating that the Gromov-Hausdorff (GH)
and Sormani-Wenger Intrinsic Flat (SWIF) limits of sequences of such
metric spaces agree.   Thus in particular the limit spaces we get are
countably $\mathcal{H}^n$ rectifiable spaces.  From these we derive GH compactness theorems for sequences of Riemannian manifolds with boundary where both the GH and SWIF limits agree. 
 For sequences of Riemannian manifolds with boundary
 we only require nonnegative Ricci curvature, upper bounds on volume, noncollapsing conditions on the interior of the manifold and diameter controls on the level sets near the boundary.  
\end{abstract}

\maketitle{}

\newpage

\section{Introduction}

In the past few decades many important compactness theorems have been 
proven for families of compact Riemannian manifolds without boundary.  Gromov introduced the notion of Gromov-Hausdorff (GH) convergence of Riemannian manifolds to metric spaces, (X,d). He proved that the
family of manifolds with nonnegative Ricci curvature and uniformly
bounded diameter are precompact in GH sense \cite{Gromov-metric}.
Cheeger-Colding have proven many properties of the GH limits of these
manifolds including rectifiability \cite{ChCo-PartI}. The Sormani-Wenger Intrinsic Flat (SWIF) convergence of oriented Riemannian manifolds to countably $\mathcal{H}^n$
rectifiable metric spaces called integral current spaces, $(X,d,T)$,
was introduced in \cite{SorWen2}. They proved that when the sequence of manifolds is noncollapsing and has nonnegative Ricci curvature the SWIF and GH limits agree \cite{SorWen1}. In general GH and SWIF limits need not agree and GH limits need not be countably
$\mathcal{H}^n$ rectifiable metric
spaces (cf. the appendix of \cite{SorWen1} by Schul and Wenger).

Here we prove GH and SWIF compactness theorems for oriented Riemannian manifolds with boundary.  Note that there are sequences of flat manifolds with boundary of bounded diameter with volume bounded below which have no GH limit, see Example \ref{ex-multipleSplines}. Nevertheless,  Wenger proved that 
a sequence of $n$-dimensional oriented Riemannian manifolds $M_j$ with boundary that 
satisfy 
\be
\diam(M_j) \leq D,\,\,\,\vol(M_j)\leq V,\,\,\vol(\bdry M_j) \leq A 
\ee
has a SWIF convergent subsequence \cite{Wenger-compactness}
(cf. \cite{SorWen2}).  Knox proved weak $L^{1,p}$ and $C^{1,\alpha}$ convergence of Riemannian manifolds with two sided bounds on the sectional curvature of the manifolds and of their boundaries, a lower bound on the volume of the boundaries and 
two sided bounds on the mean curvature of the boundary \cite{Knox-2012}. Wong proved GH convergence of Riemannian manifolds with Ricci curvature bounded below and two sided bounds on the second fundamental form of the boundaries \cite{Wong-2008}. Under stronger conditions Anderson-Katsuda-Kurylev-Lassas-Taylor \cite{Anderson-2004} and Kodani \cite{Kodani-1990} have respectively proven $C^{1,\alpha}$ and Lipschitz compactness theorems.  

We first prove compactness theorems for sequences of metric spaces and from them we derive compactness theorems for sequences of Riemannian manifolds with boundary where both the GH and SWIF limits agree. Thus we produce countably $\mathcal{H}^n$ rectifiable GH limit spaces. For sequences of Riemannian manifolds we only require Ricci curvature bounds, noncollapsing conditions on the interior of the manifold and additional controls on the boundary. 

To precisely state our theorems we recall a few notions. Let $(M,g)$ be a Riemannian manifold with boundary, $\partial M$. We denote by $d$ the metric on $M$ induced by $g$. For $\delta >0$ we define the $\delta$-inner region of $M$ by
\be 
M^{\delta} =\{ \, x\in M : d(x,\bdry M) >\delta\}.
\ee
There are two metrics on $M^\delta$. The restricted metric $d|_M$ (that we denote by $d$ to simplify notation) and the length metric $d_{M^\delta}$ induced by $g$.
The diameter of $M^\delta$ with respect to this metric is given by
\be
\diam(M^\delta, d_{M^\delta})=\sup \left\{d_{M^\delta}(x,y): x,y \in M^\delta \right\}.
\ee

In \cite{PerSor-2013}, the author and Sormani proved a GH compactness theorem for sequences of inner regions $(M^\delta_j,d_j)$ that have nonnegative Ricci curvature, 
upper bounds on volume and diameter as in (\ref{eq-PerSorClass1}) 
and a noncollapsing condition as in
(\ref{eq-PerSorClass2}) 
(cf. Theorem \ref{thm-GHdelta} within).  Now we
add one additional condition on the boundary (\ref{eq-ConvBound})
to obtain GH convergence of the sequence of manifolds themselves:   

\begin{thm}\label{thm-GHconvergence2}
Let $n\in \N$, $\delta,D_i,V,\theta >0$ and $\{\delta_i\} \subset \R$ be a decreasing sequence that converges to zero. Let $(M_j,g_j)$ be a sequence of compact oriented manifolds with boundary such that
\be\label{eq-PerSorClass1} 
\ric(M_j) \geq 0, \,\,\,  \vol(M_j)\le V, \,\,\,
\diam(M_j^{\delta_i}, {\,d_{M_j^{\delta_i}}}) \leq D_{i}, 
\ee
\be\label{eq-PerSorClass2} 
\exists q\in M_j^\delta\,\, \text{such that}\,\,\vol(B(q,\delta)) \ge \theta\delta^n, 
\ee
where $B(q,\delta)$ is the ball in $M_j$ with center $q$  and radius $\delta$, and 
suppose that there is a compact metric space $(X_\bdry,d_\bdry)$ such that
\be\label{eq-ConvBound} 
(\bdry M_j,d_j) \GHto (X_\bdry,d_\bdry).
\ee
Then a subsequence of $\{(M_j,d_j)\}_{j=1}^{\infty}$ converges in GH sense.
\end{thm}

In Example~\ref{ex-oneSplineInside} we define a sequence $(M_j,d_j)$ with no GH converging subsequence that satisfies all the conditions of Theorem~\ref{thm-GHconvergence2} except the
Ricci lower bound. Note that, by Gromov's Embedding Theorem, 
if $(M_j,d_j)$ converges in GH sense then a subsequence of 
$(\bdry M_j,d_j)$ converges in GH sense.  
In Example \ref{ex-multipleSplines} we define a sequence $(M_j,d_j)$ with no GH converging subsequence that satisfies all the conditions
of Theorem~\ref{thm-GHconvergence2} except that $(\bdry M_j,d_j) $ does not have any GH convergent subsequence.    In Theorem \ref{thm-bdryConv} we obtain convergence of the boundary as in (\ref{eq-ConvBound}) by requiring uniform bounds on the second fundamental form of $\bdry M_j$ and its derivative in the normal direction.

Suppose that $\{(M_j,g_j)\}$ satisfies the hypotheses of Theorem \ref{thm-GHconvergence2} so that we
have a subsequence such that 
$
(M_j, d_j) \GHto (X,d_X) .
$ 
Sormani-Wenger proved that for such a sequence, if for all $j$
\be \label{V-A-conditions}
\vol(M_j) \leq V \textrm{ and } \vol(\bdry M_j) \leq A 
\ee
then there exists a subsequence and an integral current space $(Y,d_Y, T)$ such that 
\be 
(M_{j_k}, d_{j_k}) \GHto (X,d_X) \,\, \text{and}\,\, (M_{j_k}, d_{j_k}, T_{j_k}) \Fto (Y, d_Y, T),
\ee 
where either $Y \subset X$ or $(Y, d_Y, T)$ is the zero integral current space \cite{SorWen2} (cf. Theorem \ref{thm-FlatInGH}).  In \cite{SorWen1} , they proved that 
for a sequence $(M_j,g_j)$ of oriented compact $n$-dimensional Riemannian manifolds with no boundary, with nonnegative Ricci curvature and with
\be
0 <v \leq \vol(M_j) \leq V,
\ee 
the GH and SWIF limits agree, $Y=X$ (cf. Theorem \ref{thm-SWgh=if}). They proved this by showing that the GH limit, $X$, is contained in a nonzero SWIF limit, $Y$,
using work of Cheeger-Colding \cite{ChCo-PartI}, Colding \cite{Colding-volume} and Perelman \cite{Perelman-max-vol}.   

In this paper we prove the corresponding theorem for manifolds with boundary.  We assume the same hypotheses as in Theorem \ref{thm-GHconvergence2} (\ref{eq-PerSorClass1-2})-(\ref{eq-ConvBound-2}) and one additional area bound on the boundary (\ref{A-condition-only}):

\begin{thm}\label{thm-GH=IF}
Let $n\in \N$, $\delta,D_i,V,\theta >0$ and $\{\delta_i\} \subset \R$ with $\delta_i$ decreasing to $0$. Let $(M_j,g_j)$ be a sequence of compact oriented manifolds with boundary such that
\be\label{eq-PerSorClass1-2} 
\ric(M_j) \geq 0, \,\,\,  \vol(M_j)\le V, \,\,\,
\diam(M_j^{\delta_i}, {\,d_{M_j^{\delta_i}}}) \leq D_{i}, 
\ee
\be\label{eq-PerSorClass2-2} 
\exists q\in M_j^\delta\,\, \text{such that}\,\,\vol(B(q,\delta)) \ge \theta\delta^n, 
\ee
and suppose that there is a compact metric space $(X_\bdry,d_\bdry)$ such that
\be\label{eq-ConvBound-2} 
(\bdry M_j,d_j) \GHto (X_\bdry,d_\bdry).
\ee
In addition, if for all $j$ we have 
\be\label{A-condition-only}
\vol(\bdry M_j) \leq A.
\ee
Then there is a subsequence that converges in SWIF sense to  a non zero integral current space:
\be
(M_{j_k}, d_{j_k}, T_{j_k}) \Fto (Y \subset X,d_X,T),
\ee
where $(X,d_X)$ denotes the GH limit of  $\{M_{j_k}, d_{j_k}\}_{k=1}^\infty$ such that 
\be
X \setminus X_\bdry \subset Y. 
\ee
If the GH limit of the boundaries is contained in the SWIF limit of the manifolds (or in the completion),
\be
\partial M_j \GHto X_{\partial} \subset Y,
\ee 
then $X=Y$ ($X=\bar Y$). Thus, $X$ is countably $\mathcal{H}^n$ rectifiable, where $\mathcal{H}^n$ denotes $n$-Hausdorff measure.  
\end{thm}

In Example \ref{ex-oneSpline} we construct a 
sequence of manifolds which have GH and SWIF limits that do not agree and that satisfies all the conditions of Theorem \ref{thm-GH=IF}, except $X_{\partial} \subset Y$.    In Example \ref{ex-oneSplineInside} we construct a 
sequence of manifolds which have GH and SWIF limits that do not agree and that satisfies all the conditions of Theorem \ref{thm-GH=IF}, except the Ricci bound.

In \cite{Anderson-2004}, Anderson-Katsuda-Kurylev-Lassas-Taylor prove that a sequence of manifolds with bounds on various injectivity radii and on diameter as well as two sided Ricci curvature bounds on the manifolds and their boundaries, one has a subsequence converging in the 
$C^{1,\alpha}$ sense.   Removing half of these hypotheses we can prove GH and SWIF convergence of the manifolds (Proposition~\ref{prop-injectivity}).   Recall that the boundary injectivity radius of $p \in \bdry M$ is defined by 
\be
i_\bdry(p)= \inf\{t \,|\, \gamma_{p} \,\text{stops minimizing at t}\},
\ee
where $\gamma_p$ is the geodesic in $M$ such that $\gamma_p'(0)$  is the inward unitary normal tangent vector at $p$. The boundary injectivity radius of $M$ is defined by 
\be
i_\bdry(M)=\inf\{i_\bdry(p)\,|\,p \in \bdry M\}.
\ee
Inside we prove that a uniform bound on the boundary injectivity radius of the manifolds implies that
the diameter bounds in (\ref{eq-PerSorClass1}) can be reduced to a single diameter bound.  In addition, we prove GH convergence of the boundaries, (\ref{eq-ConvBound}), and $X_\partial \subset \bar{Y}$.  Combining this with our theorems above we obtain:

\begin{thm}\label{thm-injectivity}
Let $n\in \N$ and $\delta,D,V,\theta, \iota >0$, $\iota > \delta$. Suppose that $(M_j,g_j)$ is a sequence of $n$-dimensional compact oriented manifolds with boundary that satisfy 
\be\label{eq-RicVolDiam}
\ric(M_j) \geq 0, \,\, \vol(M_j)\le V,\,\, \vol(\bdry M_j) \leq A\,\,\diam(M_j^\delta,{d_{M_j^\delta}})\le D,
\ee
\be\label{eq-theta} 
\exists q\in M_j^\delta\,\, \text{such that}\,\,\vol(B(q,\delta)) \ge \theta\delta^n,
\ee
where $B(q,\delta)$ is the ball in $M_j$ with center $q$  and radius $\delta$ and
\be\label{eq-bouInjrad} 
\,i_\bdry(M_j) \geq \iota.
\ee
Then there is a subsequence such that
\be 
(M_{j_k}, d_{j_k}) \GHto (X,d_X),
\ee 
where $X$ is the closure of an $\mathcal H^n$ countably rectifiable metric space. 
\end{thm}

In Example \ref{ex-oneSpline} we construct a sequence of manifolds which has a GH limit $X$ 
which is not the closure of a countably rectifiable metric space. The sequence satisfies all the conditions of Theorem \ref{thm-injectivity}, except that the boundary injectivity radii do not have a positive uniform lower bound.  In Example \ref{ex-injFail} we construct a sequence of manifolds that satisfies all the conditions of Theorem \ref{thm-injectivity}. 

Our results concerning sequences of Riemannian
manifolds with boundary are consequences of 
the next two theorems concerning
sequences of metric spaces and integral current spaces. Let $(X,d)$ be a metric space. In \cite{PerSor-2013}, the author and Sormani defined the boundary of a metric space to be 
\be\label{defn-bdry}
\bdry X= \bar{X} \setminus X,
\ee
 where $\bar{X}$ is the metric completion of $X$. This agrees with the notion of the boundary of a manifold with boundary if one takes $X$ to be interior of the manifold.    
 
 For $\delta >0$ let 
\be 
X^{\delta} :=\{ \, x\in X \,|\, d(x,\bdry X) >\delta\}
\ee
be the $\delta$-inner region of $X$. 

In prior work of the author with Sormani \cite{PerSor-2013}, applying
Gromov's Embedding Theorem \cite{Gromov-groups}, it was proven  that given a sequence $\{\delta_i\}$ decreasing to zero and a sequence of compact metric spaces with boundary that converge in $GH$ sense, $(X_j, d_j) \GHto (X,d_X)$, there is a subsequence
$\{j_k\}$ and compact subspaces 
\be
X(\delta_i) \subset X \textrm{ and } X_\bdry \subset X
\ee 
such that 
\be 
(\overline{X^{\delta_i}_{j_k}}, d_{j_k}) \GHto (X(\delta_i),d_X)
\ee
for all $i$ and
\be 
(\bdry X_{j_k}, d_{j_k}) \GHto (X_\bdry,d_X). 
\ee
In Theorem \ref{thm-GHconvergence1} we prove the converse:

\begin{thm}\label{thm-GHconvergence1}
Let $(X_j,d_j)$ be a sequence of precompact length metric spaces with precompact boundary.  Suppose that there is a decreasing sequence  $\{\delta_i\}_{i=1}^\infty \subset \R$ that converges to zero such that $X_j^{\delta_i} \neq \emptyset$ and
$\{(\overline{X_j^{\delta_i}},d_j)\}_{j=1}^{\infty}$ 
converges in GH sense for all $i$ to some compact metric space $(X(\delta_i),d_{X(\delta_i)})$,
\be \label{eq-Xdeltaconv} 
(\overline{X_j^{\delta_i}},d_j) \GHto (X(\delta_i),d_{X(\delta_i)}).
\ee
Suppose that there is a compact metric space $(X_\bdry,d_\bdry)$ such that
\be 
(\bdry X_j,d_j) \GHto (X_\bdry,d_\bdry).
\ee 
Then a subsequence of $\{(\bar X_j,d_j)\}_{j=1}^{\infty}$ converges in GH sense.
\end{thm}

The next theorem is the key ingredient to prove our theorems in which 
both GH and SWIF limits agree.  In particular it is
applied to prove Theorem \ref{thm-GH=IF}  and Theorem \ref{thm-injectivity}.

\begin{thm}\label{thm-generalizeConv1}
Let $( X_j,d_j, T_j)$ be precompact integral current spaces. Suppose that there exist a compact metric space $(X,d)$  and a non zero integral current space $(Y\subset X,d,T)$ 
such that 
\be 
(\bar X_j, d_j) \GHto (X,d) \,\, \text{and}\,\, (X_j, d_j, T_j) \Fto (Y,d,T).
\ee
and there is a subsequence such that
\be 
(\partial X_{j_k}, d_{j_k}) \GHto (X_\partial ,d)
\ee
where $\partial X_{j_k}$ is defined as in (\ref{defn-bdry}).
Suppose in addition that
\be\label{eq-completionY}
X_{\partial} \subset \bar{Y},
\ee
and
there is a decreasing sequence $\delta_i\to 0$ such that the inner regions converge
 \be
(\overline{X_{j_k}^{\delta_i}}, d_{j_k}) \GHto X(\delta_i) 
\ee
with
\be\label{inside-inside}
X(\delta_i) \subset Y \qquad \forall i \in \mathbb{N}.
\ee
Then $X=\bar{Y}$.
\end{thm}

In Example \ref{ex-oneSplineInside} we construct a 
sequence of manifolds which have GH and SWIF limits that do not agree that satisfies all the conditions of Theorem \ref{thm-generalizeConv1}, except that $X^{\delta} \subset Y$. In Example \ref{ex-oneSpline} we describe a sequence of regions in Euclidean space with GH and SWIF limits that do not agree. This sequence satisfies all the conditions of 
Theorem \ref{thm-generalizeConv1}, except $X_\bdry \subset Y$.  

We now provide an outline for the paper. We begin with two sections reviewing key theorems needed to prove the results in this paper.  In Section \ref{sec-GH} we review GH convergence
as defined in Gromov's book \cite{Gromov-metric}. We state prior 
GH convergence results for manifolds with boundary proven by Sormani and the author \cite{PerSor-2013}. In Subsection \ref{sbs-ChCo} we state Colding's volume estimate for balls GH close to balls in Euclidean space \cite{Colding-volume}, and some of Cheeger-Colding's results about GH limits of non collapsed sequences of Riemannian manifolds with Ricci curvature bounded below \cite{ChCo-PartI}.  

In Section \ref{sec-IF} we go over Ambrosio-Kirchheim's  
results concerning integral currents \cite{AK}. We then review Sormani-Wenger's integral current spaces, SWIF distance and some of their theorems. We present a simplified proof of Sormani-Wenger's GH=SWIF theorem for manifolds with no boundary \cite{SorWen1} (cf. Theorem \ref{thm-SWgh=if}).  This simplified proof will be adapted later to prove Theorem \ref{thm-GH=IF}.

In Section \ref{sec-ConvergenceMet}, we prove the new convergence theorems for metric spaces stated above: Theorem \ref{thm-GHconvergence1} and Theorem \ref{thm-generalizeConv1}.   We also present some important examples.  There we see that the sequence of metric spaces described in Example \ref{ex-multipleSplines} satisfies all the conditions of Theorem\ref{thm-GHconvergence1}, except the GH convergence of the boundaries. 
Meanwhile, the sequence of metric spaces described in Example \ref{ex-multipleSplines2} satisfies all the conditions of Theorem\ref{thm-GHconvergence1}, except that there is no GH convergent subsequence of inner regions $(M^\delta_j, d_j)$
for any $\delta$ small. In both examples the conclusion of 
Theorem\ref{thm-GHconvergence1} does not hold.  To prove the importance of our conditions in Theorem\ref{thm-generalizeConv1}, we present two examples. In Example \ref{ex-oneSplineInside} we describe a sequence for which the GH limit of the sequences of inner regions is not contained in the SWIF limit and in Example \ref{ex-oneSplineInside} we show a sequence for which the GH and SWIF limit do not agree since the GH limit of the boundaries is not contained in the SWIF limit. 

In Section \ref{sec-ConvergenceMet} we also prove Theorem \ref{thm-GHCollapsing} which deals with the case when the GH limit of $(X_j,d_j)$ agrees with the GH limit of $(\bdry X_j,d_j)$. In Example \ref{ex-collapsingCylinders} we describe a sequence of 3-dimensional cylinders $X_j \subset \R^3$ such that $\{X_j\}$ and $\{\bdry X_j\}$ GH converge to a segment. Notice that if the Hausdorff dimension of the GH limit drops then the GH limit cannot agree with the SWIF  limit. In Lemma \ref{lem-IntPtX} we characterize the points of $X \setminus X_\bdry$.

In Section \ref{sec-ConvergenceMan} we prove our new theorems about limits of Riemannian manifolds with boundary (Theorem \ref{thm-GHconvergence2}, Theorem \ref{thm-GH=IF}
and Theorem \ref{thm-injectivity}) and present key examples related to these theorems.  In Subsection \ref{sbs-GHconvergence2} we prove Theorem \ref{thm-GHconvergence2} by applying Theorem \ref{thm-GHconvergence1}. We note that Example \ref{ex-oneSpline} satisfies all the conditions of Theorem \ref{thm-GHconvergence2}, hence it has a GH limit.

In Subsection \ref{sbs-GH=IF} we prove Theorem \ref{thm-GH=IF} by adapting the simplified proof of Sormani-Wenger's GH=IF Theorem for manifolds with no boundary \cite{SorWen2}.  See Section \ref{sbs-SWgh=if} for details of their proof.  We also see that the sequence described in Example \ref{ex-oneSpline} satisfies all the conditions of Theorem \ref{thm-GH=IF}, except Equation (\ref{eq-completionY}) and so the GH limit does not coincide with the SWIF limit.  Then we provide two examples, Example \ref{ex-bdryGHandIF} and Example \ref{ex-bdryGH=IF}, in which the sequences satisfy all the conditions of Theorem \ref{thm-GH=IF}. Additionally, in these two examples the GH limit of the sequences of boundaries, $X_\bdry$, do not agree with the SWIF limit of the boundaries, $(Y_\bdry, d, T_\bdry)$, showing that we cannot replace Equation (\ref{eq-completionY}) by $X_\bdry=Y_\bdry$.  

In Subsection \ref{sbs-injectivity} we prove Theorem \ref{thm-injectivity}. To prove it we will apply Theorem \ref{thm-GH=IF}.  To do so, we first prove uniform diameter bounds for sequences of inner regions, Lemma \ref{lem-diameters}, and the GH convergence of the sequence of boundaries, Lemma \ref{thm-bdryConv}. Then we show that $X_\bdry \subset \bar Y$. From Example \ref{ex-oneSpline} we notice that a positive lower bound on the injectivity radii is not necessary for the GH convergence. In Example \ref{ex-injFail} we present a sequence that satisfies all the conditions of Theorem \ref{thm-injectivity}, except the positive boundary injectivity radii bound. In this example the GH limit and the SWIF limit agree showing that the hypothesis in the boundary injectivity radii is stronger than necessary. 

In Section \ref{sec-ConvergenceBoun} we prove GH convergence of sequences of boundaries, Theorem \ref{thm-bdryConv}.  
In order to prove this theorem, in Proposition  \ref{prop-epsCovers} we show that the GH convergence of $(\bdry M_j, d_{\bdry M_j})$ implies the convergence of $(\bdry M_j, d_j)$ and in Proposition \ref{prop-RiciBounded} we obtain a uniform Ricci curvature bound on the boundaries. Notice that in Theorem \ref{thm-GHconvergence2} and Theorem \ref{thm-GH=IF} one of the hypothesis is the convergence in GH sense of the sequence $(\bdry M_j, d_j)$, Equation (\ref{eq-ConvBound}) and Equation (\ref{eq-ConvBound-2}). So Theorem \ref{thm-bdryConv} can be used in these cases. 



I would like to thank my doctoral advisor, Professor Sormani, for introducing me to these notions of convergence and helping with expository aspects of this paper. I also would like to thank Professors Anderson, Lawson and Fukaya for their excellent courses and their support. 

I would like to thank Professor Villani for inviting me to present at the Optimal Transport Reading Seminar at MSRI and Professor Gigli for his excellent courses on this subject at MSRI and at Bonn as well.

I want to thank Professors Searle, Plaut and Wilkins for providing me with the opportunity to speak at the Smoky Cascade Geometry Conference at Knoxville.  I would like to thank Notre Dame University for providing me with the opportunity to speak at the Felix Klein Seminar. I want to thank Professor Nabutovsky for inviting me to give a talk at the University of Toronto. I thank the American Mathematical Society that provided me with the opportunity to talk at the Spring Southwest Sectional Meeting in 2014. 
   
\tableofcontents


\section{A Review of GH Limits}\label{sec-GH}

In this section we list GH convergence results that will be used in the next sections
including prior published results of the author with Sormani as well as work of
Gromov and Cheeger-Colding.
In Subsection \ref{sbs-GHC} we define Gromov-Hausdorff distance and state Gromov's compactness theorem and its converse; Theorem \ref{thm-GromovCompactness} and \ref{thm-ConverseGromov}. In subsection \ref{sbs-PerSor} we review the GH compactness theorems for $\delta$-inner regions of manifolds with boundary proven by the author and Sormani in \cite{PerSor-2013}, Theorem \ref{thm-GHdelta} and \ref{thm-GHdeltas}.   
In Subsection \ref{sbs-ChCo} we state Colding's theorem about the volume of balls being close to the volume of balls in $\R^n$, provided the balls are GH close and the Ricci curvature is bounded below \cite{Colding-volume}, cf. Theorem \ref{thm-Colding-volume}. We also state Cheeger-Colding's theorem \cite{ChCo-PartI} cf. Theorem \ref{thm-RegularSet} about the singular set of the GH limit of a noncollapsing sequences of Riemannian manifolds with curvature bounded below having zero Hausdorff measure and the regular points having all tangent cones of the maximal dimension. These results will be applied to prove that the GH limit agrees with the SWIF limit in Theorem \ref{thm-GH=IF}.

\subsection{Gromov-Hausdoff Convergence}\label{sbs-GHC}
Here we introduce Gromov-Hausdorff convergence.  A detailed exposition see Burago-Burago-Ivanov \cite{BBI} and Gromov \cite{Gromov-metric}.

The Hausdorff distance in a complete metric space $Z$, $d^Z_H$, between two subsets $A,B \subset Z$ is defined as
\be
d_{H}^Z\left(A,B\right) = \inf\{ \varepsilon>0: A \subset T_\varepsilon\left(B\right) \textrm{ and } B \subset T_\varepsilon\left(A\right)\}.
\ee
Here, $T_\varepsilon(A)$ denotes the $\varepsilon$ neighborhood of $A$.

\begin{defn}[Gromov]\label{defn-GH} 
Let $\left(X_i, d_{X_i}\right)$, $i=1,2$, be two metric spaces. The Gromov-Hausdorff distance between them is defined as
\be \label{eqn-GH-def}
d_{GH}\left(X_1,X_2\right) = \inf  \, d^Z_H\left(\varphi_1\left(X_1\right), \varphi_2\left(X_2\right)\right)
\ee
where $Z$ is a complete metric space and $\varphi_i: X_i \to Z$ are distance preserving maps. 
\end{defn}

The above function is symmetric and satisfies the triangle inequality.
It is a distance when considering compact metric spaces.

\begin{thm}[Gromov]\label{thm-GromovCompactness}
Let $\left(X_j, d_j\right)$ be a sequence of compact metric spaces. If there exist $D$ and $N:(0,\infty) \to \N$  such that for all $j$ 
\be   
\diam(X_j) \leq D
\ee
and for all $\varepsilon$ there are $N(\varepsilon)$  $\varepsilon$-balls that cover $X_j$, then there exist a compact metric space $(X,d_X)$ and a subsequence such that 
\be 
\left(X_{j_k}, d_{j_k}\right) \GHto \left(X,d_X\right).
\ee 
\end{thm}

\begin{defn}\label{defn-equibounded}
We say that a family, $\mathcal{F}$, of compact metric spaces is equibounded if there exists a function $N:(0,\infty) \to \N$ as in theorem \ref{thm-GromovCompactness}. For the purpose of clarity we will denote $N$ by $N(\,\cdot \, , \mathcal{F})$ when working with different families.
\end{defn}

\begin{thm}[Gromov]\label{thm-ConverseGromov}
Let $\{(X_j, d_j)\}$ be a sequence of compact metric spaces that converges in GH sense. Then $\{(X_j, d_j)\}$ is equibounded in the sense of Definition \ref{defn-equibounded} and, there is $D>0$ such that 
$\diam(X_j) \leq D $ for all $j$.
\end{thm}

\begin{thm}[Gromov in \cite{Gromov-81a}]\label{thm-GromovEmbedding}
Let $\{(X_j, d_j)\}$ be a sequence of compact metric spaces that converges in GH sense
to $(X_\infty,d_\infty)$. Then there is a compact metric space $(Z,d)$ and isometric embeddings $\varphi_j:X_j \to Z$ such that a subsequence $\{(\varphi_{j_k}(X_{j_k}),d)\}$ converges in Hausdorff sense to $(\varphi_\infty(X_\infty),d)$. 
\end{thm}

Whenever we have a GH converging sequence, we choose embeddings $\varphi_j$ as in the previous theorem. Then we consider $\{(X_{j_k},d_{j_k})\}$ to be our original sequence, $\{(X_j, d_j)\}$. We say that a sequence $x_j \in X_j$ converges to $x_\infty \in X_\infty$ if 
\be  
\varphi_j(x_j) \to \varphi_\infty(x_\infty).
\ee  

Moreover, using the following theorem we can say that a sequence $A_j \subset X_j$ GH converges to a set $A_\infty \subset X_\infty$. 

\begin{thm}[Blaschke]\label{thm-Blaschke}
Let $(Z,d)$ be a compact metric space and $A_j$ be a sequence of closed subsets of $Z$. Then, there is a subsequence $A_{j_k}$ that converges in Hausdorff sense. 
\end{thm}

For Riemannian manifolds with no boundary the following compactness theorem holds.  

\begin{thm}[Gromov]\label{thm-GromovRic}
Every sequence of $n$-dimensional compact Riemannian manifolds with diameter $\leq D$ and $\ric\geq (n-1)k$ has a GH convergent subsequence.
\end{thm}

\subsection{Prior GH Convergence Results of the Author with Sormani} \label{sbs-PerSor}

In this subsection we review results published in \cite{PerSor-2013}.

Recall that for a Riemannian manifold with boundary $(M,g)$ the $\delta$-inner region of $M$ is given by 
\be
M^{\delta} =\{ \, x\in M : d(x,\bdry M) >\delta\}.
\ee
The inner regions may be endowed with the induced length metric $d_{M^\delta}$ 
\be
d_{M^\delta}(x,y) := \inf\Big\{ L_g(C): \,\, C:[0,1]\to M^\delta \,\,\text{smooth}, \, C(0)=x, \, C(1)=y \Big\}
\ee
(which is possibly infinite) or the restricted metric $d$
\be 
d(x,y) := \inf\Big\{ L_g(C): \,\, C:[0,1]\to M \,\,\text{smooth}, \, C(0)=x, \, C(1)=y \Big\}
\ee
where
\be
L_g(C)=\int_0^1 g^{1/2}(C'(t),C'(t))\, dt.
\ee

In the first theorem presented here, we proved GH subconvergence of the sequence $M^\delta_j$ with respect to the restricted metric since this provides more information about the original sequence of manifolds $M_j$.  But note that the diameter bound we required is with respect to the induced length metric.   In particular the inner regions were assumed to be path connected:  

\begin{thm}[P\---Sormani]\label{thm-GHdelta}
Given $n\in \N$ and $\delta,D,V,\theta >0$ suppose that $(M_j, g_j)$ is a sequence of compact oriented manifolds with boundary such that
\be
\ric(M_j) \geq 0, \,\,\,  \vol(M_j)\le V, \,\,\,
\diam(M_j^{\delta}, {\,d_{M_j^{\delta}}}) \leq D, 
\ee
\be
\exists q\in M_j^\delta\,\, \text{such that}\,\,\vol(B(q,\delta)) \ge \theta\delta^n, 
\ee
where $B(q,\delta)$ is the ball in $M_j$ with center $q$  and radius $\delta$. Then there is a subsequence $\{j_k\}$ and a compact metric
space $(X(\delta),d_{X(\delta)})$ such that 
\be
(\overline{M^\delta_{j_k}}, d_{M_{j_k}}) \GHto (X(\delta),d_{X(\delta)}).
\ee
\end{thm}  

\begin{rmrk}\label{rmrk-VolBall}
In the proof of Theorem \ref{thm-GHdelta} it was shown that for all $p \in M^\delta$ and 
$\varepsilon < \delta/2$ 
\be
\vol(B(p,\varepsilon)) \geq  2^{-nD/\varepsilon} \theta\varepsilon^n.
\ee
This estimate also works for $\varepsilon=\delta/2$. Choosing $\varepsilon=\delta/2$ and applying Bishop-Gromov Volume Comparison we get:
\be\label{eq-noncollapsing}
\vol(B(p,r)) \geq  2^{-nD/(\delta/2)} \theta r^n.
\ee
for all $r \leq \delta/2$.
\end{rmrk}

For a decreasing sequence of real numbers, $\delta_i \to 0$, we obtained simultaneous convergence of sequences of inner regions.

\begin{thm}[P\---Sormani]\label{thm-GHdeltas}
Take $n\in \N$, a decreasing sequence, $\delta_i \to 0$, $D_i>0$, $i=0,1,2...$, $V>0$ and $\theta>0$. Suppose that $(M_j, g_j)$ is a sequence of compact $n$-dimensional Riemannian manifolds with boundary such that
\be
\ric(M_j) \geq 0, \,\,\,  \vol(M_j)\le V, \,\,\,
\diam(M_j^{\delta_i}, {\,d_{M_j^{\delta_i}}}) \leq D_{i} \,\,\,\forall i 
\ee
and 
\be
\exists q\in M_j^\delta\,\, \text{such that}\,\,\vol(B(q,\delta)) \ge \theta\delta^n, 
\ee
where $B(q,\delta)$ is the ball in $M_j$ with center $q$  and radius $\delta$. Then there is a subsequence $\{j_k\}$ such that 
$(\overline{M^{\delta_i}_{j_k}}, d_{M_{j_k}})$ converges in Gromov-Hausdorff sense for all $i$.
\end{thm}

\subsection{Cheeger-Colding Theorems}\label{sbs-ChCo}

Here we review a result by Colding \cite{Colding-volume} and few of the many important theorems of Cheeger-Colding proven in \cite{ChCo-PartI} that we need to prove that the GH limit is inside the SWIF limit, see proof of Theorem \ref{thm-GH=IF}. 

The next theorem tells us that the volume of balls of manifolds are close to the volume of balls in Euclidean space when these balls are close in GH sense. This result is used in Sormani-Wenger \cite{SorWen1} to prove that the GH limit of manifolds with no boundary coincides with the SWIF limit (cf Theorem \ref{thm-SWgh=if} within). We will use this theorem as well to prove our new Theorem \ref{thm-GH=IF}.  

\begin{thm}[Colding, Corollary 2.19 in \cite{Colding-volume}]\label{thm-Colding-volume}
For all $\varepsilon >0$ and $n \in \N$ there exist $k(\varepsilon,n) > 0$ and $\delta(\varepsilon,n) >0 $ such that for any complete n-dimensional Riemannian manifold $M$ that satisfies 
\be  
\ric(M) \ge -(n-1)k \text{ and  } d_{GH}(B(p,1),B(0,1)) < \delta,
\ee
the following holds:
\be 
|\vol(B(p,1)) -\vol(B(0,1))| < \varepsilon,
\ee
where $B(0,1)$ denotes the open ball of radius $1$ and center $0$ in the Euclidean space $\R^n$.
\end{thm}

In the noncollapsing case the volume of the manifolds converge to the Hausdorff measure of the limit space. 

\begin{thm}[Cheeger-Colding \cite{ChCo-PartI}] \label{thm-chco}
Let $k \in \R$, $v>0$ and $\{M_j^n\}$ be a sequence of $n$-dimensional compact Riemannian manifolds
such that
\be
\ric(M_j)\geq (n-1)k,\,\, M_j \GHto X \textrm{ and } \vol(M_j) \ge v.
\ee
Then for all $r>0$ and  $x \in X$
\be
\lim_{j\to\infty} \vol(B(x_j,r))= \mathcal{H}^n(B(x,r)),
\ee 
where $x_j\in M_j$ such that $x_j \to x$ and $\mathcal{H}^n$ denotes $n$-Hausdorff measure. 
In particular, 
\be
\lim_{j\to\infty} \vol(M_j)= \mathcal{H}^n(X).
\ee
\end{thm}

\begin{rmrk}\label{rmrk-singSet}
Since the theorem is proven locally, if $M_j$ is a sequence of $n$-dimensional manifolds with boundary that satisfy 
\be
\ric(M_j)\geq (n-1)k,\,\, M_j \GHto X
\ee
and for each $x \in M_j^\delta$
\be
\vol(B(x,r)) \ge v(\delta)>0
\ee
for $r\leq \delta/2$. Then 
\be
\lim_{j\to\infty} \vol(B(x_j,r))= \mathcal{H}^n(B(x,r)), 
\ee 
where $x \in X$, $x_j\in M^\delta_j$ such that $x_j \to x$ and $r \leq \delta/2$.  
\end{rmrk}

\begin{defn}
A sequence $\{(X_j, d_j,p_j) \}$, $p_j \in X_j$, converges in the pointed Gromov-Hausdorff sense to a metric space $(X,d,p)$ if the following holds. For all $r >0$ and $\varepsilon >0$ there exists $N \in \mathbb N$ and maps 
\be f_{j}:B(p_j,r) \to X
\ee
such that
\be 
f(p_j)=p,\,\,\,\,d_{GH}(B(p_j,r),f_j(B(p_j,r))) < 2\delta 
\ee
and 
\be 
B(p,r-\varepsilon) \subset  T_\varepsilon f_j(B(p_j,r)),
\ee
where $T_\varepsilon f_j(B(p_j,r))$ is the $\varepsilon$ neighborhood of $f_j(B(p_j,r))$.
\end{defn}

\begin{defn}
Let $(X,d)$ be a metric space. A tangent cone at $x \in X$ is a complete pointed GH limit $(X,d_\infty,x)$ of a sequence of the form $\{(X,r_j^{-1}d,x)\}$, where $\lim_{j \to \infty}r_j=0$.
\end{defn}

\begin{defn}\label{defn-regpt}
A point $x \in X$ is called regular if for some $k$ every tangent cone of $x$ is isometric to $\R^k$. A point is called non regular  if it is not regular.  The set of regular points of $X$ is denoted by $\mathcal R (X)$.
\end{defn}

\begin{thm}[Cheeger-Colding, Theorem 2.1 and Theorem 5.9 \cite{ChCo-PartI}]\label{thm-RegularSet}
Let $k \in \R$, $v>0$ and $\{M_j^n\}$ be a sequence of $n$-dimensional compact Riemannian manifolds
such that
\be
\ric(M_j)\geq (n-1)k,\,\, M_j \GHto X \textrm{ and } \vol(M_j) \ge v.
\ee
Then the set of nonregular points of $X$ has zero $n$-Hausdorff measure and all the tangent cones of the regular points of $X$ are isometric to $\R^n$.
\end{thm}

\begin{rmrk}\label{rmrk-singSet}
Since the theorem is proven locally, if $M_j$ is a sequence of $n$ dimensional manifolds with boundary that satisfy 
\be
\ric(M_j)\geq (n-1)k,\,\, M_j \GHto X
\ee
and for each $x \in M_j^\delta$
\be
\vol(B(x,r)) \ge v(\delta)>0
\ee
for $r<\delta/2$, then the set of nonregular points of $X$ contained in 
\be
X(\delta)=\{x \in X \,|\, \exists\, x_j \in M_j^\delta \to x \} 
\ee
has zero $n$-Hausdorff measure and all the tangent cones of the regular points of $X$ contained in $X(\delta)$ are isometric to $\R^n$.

\end{rmrk}


\section{A Review of Integral Current Spaces and SWIF Convergence}\label{sec-IF}

In Subsection \ref{sbs-currents} we review the notion and properties of integral currents that appear on Ambrosio-Kirchheim's paper ``Currents in Metric Spaces" \cite{AK}. Here we see that an integral current, $T$, in a metric space is a current acting
on a tuple of functions (rather than a differential form) that has integer valued Borel weight functions whose boundaries are also integer rectifiable currents. The set of the current,
denoted $\set(T)$, is an oriented countably $\mathcal{H}^n$ rectifiable subset of the given metric space.   

In Subsection \ref{sbs-IFdistance} we see that Sormani-Wenger \cite{SorWen2} defined integral rectifiable current spaces, $(Y,d,T)$, where $T$ is an integral current in $\bar{Y}$ and $\set(T)=Y$.  We also define the Sormani-Wenger intrinsic flat distance (SWIF distance) which was defined in imitation of Gromov's intrinsic Hausdorff distance (GH distance), except that the Hausdorff distance, $d_H$, in Definition~\ref{defn-GH} is replaced by Federer-Fleming's flat distance $d_F$ \cite{SorWen2}. We end Subsection \ref{sbs-IFdistance} with Sormani-Wenger's Theorem that shows that under certain conditions the GH limit contains the SWIF limit from \cite{SorWen2}. 

In Subsection \ref{sbs-SWgh=if} we explain Sormani-Wenger's GH=IF Theorem for manifolds with no boundary, \cite{SorWen1} (cf. Theorem \ref{thm-SWgh=if} within).


\subsection{Integral Currents}\label{sbs-currents}

The aim of this subsection is to review Ambrosio-Kirchheim's notion of an integral current on a metric space (which extends the notion of Federer-Fleming) \cite{AK} \cite{FF}.
To accomplish this we define currents, Definition \ref{defn-current}, and integer currents, Definition  \ref{def-integercur}. Then we mention two important properties of integer currents proven by Ambrosio-Kirchheim \cite{AK}.  
The characterization of the mass measure, Lemma \ref{lemma-weight}, which is amply used in SWIF convergence and that the $\set$ is a countably rectifiable metric space, Lemma \ref{lemma-setrect}. The subsection finishes with the definition of an integral current, Definition \ref{defn-integralcur}.  

For a metric space $Z$, denote by $\mathcal{D}^m(Z)$ the collection
of $(m+1)$-tuples of Lipschitz functions where the first entry is a bounded function:
\be
\mathcal{D}^m(Z)=\left\{ (f,\pi)=\left(f,\pi_1 ..., \pi_m\right)\, |\, f, \pi_i: Z \to \R \text{ Lipschitz and } f \,\text{ is bounded}\right\}.
\ee

\begin{defn}[Ambrosio-Kirchheim]\label{defn-current}
Let $Z$ be a complete metric space. A multilinear functional $T:\mathcal{D}^m(Z) \to \R$ is called an $m$ dimensional current if it satisfies:

i) If there is an $i$ such that $\pi_i$ is constant on a neighborhood of $\{f\neq0\}$ then $T(f, \pi)=0$.

ii) $T$ is continuous with respect to the pointwise convergence of the $\pi_i$ for  $\Lip(\pi_i)\le 1$.

iii) There exists a finite Borel measure $\mu$ on $Z$ such that for all $(f,\pi)\in \mathcal{D}^m(Z)$
\be\label{def-AK-current-iii}
|T(f,\pi)| \le \prod_{i=1}^m \Lip(\pi_i)  \int_Z |f| \,d\mu .
\ee
The collection of all m dimensional currents of $Z$ is denoted by $\curr_m(Z)$.
\end{defn}

To each current we associate a measure and a mass: 

\begin{defn}[Ambrosio-Kirchheim] \label{defn-mass}
Let $T:\mathcal{D}^m(Z) \to \R$ be an $m$-dimensional current. The mass measure of $T$ is the smallest Borel measure $\|T\|$  such that (\ref{def-AK-current-iii}) holds for all $(f,\pi) \in \mathcal{D}^m(Z)$.

The mass of $T$ is defined as
\be \label{def-mass-from-current}
M\left(T\right) = || T || \left(Z\right) = \int_Z \, d\| T\|.
\ee
\end{defn}

To give the definition of integer current, 
we first see how to get a current by pushing forward another one, Definition \ref{defn-push}, and in Example \ref{defn-basiccur}
we define a current in Euclidean space, $\R^n$, that only requires an integer valued $L^1$ function.

\begin{defn}[Ambrosio-Kirchheim Defn 2.4]\label{defn-push}
Let $T\in \curr_m(Z)$ and $\varphi:Z\to Z'$ be a
Lipschitz map. The {\em pushforward} of $T$
to a current $\varphi_\# T \in \curr_m(Z')$ is given by
\be \label{def-push-forward}
\varphi_\#T(f,\pi)=T(f\circ \varphi, \pi_1\circ\varphi,..., \pi_m\circ\varphi).
\ee
\end{defn}

\begin{ex}[Ambrosio-Kirchheim]\label{defn-basiccur}
Let $h: A \subset \R^m \to \Z$ be an $L^1$ function. Then $\Lbrack h \Rbrack:  \mathcal{D}^m(\R^m) \to \R$ given by
\be \label{def-current-from-function}
\Lbrack h \Rbrack \left(f, \pi\right) = \int_{A \subset \R^m}  h f \det\left(\nabla \pi_i\right) \, d\mathcal{L}^m
\ee
is an $m$ dimensional current, where $\nabla \pi_i$ are defined almost everywhere by Rademacher's Theorem.
\end{ex}

Now we proceed to define integer currents: 
\begin{defn}[Defn 4.2, Thm 4.5 in Ambrosio-Kirchheim \cite{AK}] \label{def-integercur} 
Let $T\in \curr_m(Z)$. $T$ is an  integer rectifiable
current if it has a parametrization of the form $\left(\{\varphi_i\}, \{\theta_i\}\right)$, where

i) $\varphi_i:A_i\subset\R^m \to Z$ is a countable collection of bilipschitz maps
such that $A_i$ are precompact Borel measurable with pairwise disjoint images,

ii) $\theta_i\in L^1\left(A_i,\N\right)$ 
such that
\be\label{param-representation}
T = \sum_{i=1}^\infty \varphi_{i\#} \Lbrack \theta_i \Rbrack \quad\text{and}\quad \mass\left(T\right) = \sum_{i=1}^\infty \mass\left(\varphi_{i\#}\Lbrack \theta_i \Rbrack\right).
\ee
The mass measure is
\be
||T|| = \sum_{i=1}^\infty ||\varphi_{i\#}\Lbrack \theta_i \Rbrack ||.
\ee
The space of $m$ dimensional integer rectifiable currents on $Z$ is denoted by
$\intrectcurr_m\left(Z\right)$.
\end{defn}

In the next lemma we see that the mass measure of an integral current is concentrated in its $\set$.

\begin{defn}[Ambrosio-Kirchheim] \label{defn-set}
Let $T \in \curr_m(Z)$, the canonical set of $T$, denoted $\set(T)$, is 
\be
\set(T)=\{ p\in Z:\, \Theta_{*m}\left( \|T\|, p\right)>0\}
\ee 
where
\be \label{eq-set}
\Theta_{*m}\left( \|T\|, p\right):= \liminf_{r\to 0} \frac{\|T\|(B(p,r))}{\omega_m r^m}.
\ee
The function $\Theta_{*m}\left( \|T\|, p\right)$ is called the $ \|T\|$ lower density of $p$  and $\omega_m$ denotes the volume of the unit ball in $\R^m$.
\end{defn}

\begin{defn} \label{defn-rstr} 
Given $T \in \curr_m(Z)$
and $A \subset Z$ a Borel set, the restriction of $T$ to $A$ is a current,  $T \rstr A \in  \curr_m(Z)$, given by 
\be
\left( T \rstr A \right) (f, \pi)=T( \chi_A f, \pi).
\ee
where $\chi_A$ is the indicator function of $A$. 
\end{defn}

We note that the mass measure of $T \rstr A$, $||T \rstr A||$, equals $||T|||_A$. Hence, $||T \rstr A||(A)=\| T\rstr A\|(Z)=||T||(A)$.   So 
\be
\Theta_{*m}\left( \|T\|, p\right):= \liminf_{r\to 0} \frac{\|T\rstr B(p,r)\| (Z)}{\omega_m r^m}.
\ee

\begin{lem}[Ambrosio-Kirchheim]\label{lemma-weight}
Let $T \in \intrectcurr_m\left(Z\right)$ with parametrization $\left(\{\varphi_i\}, \theta_i\right)$. Then there is a function
\be \label{eqn-lem-weight-lambda}
\lambda:\set(T) \to [m^{-m/2}, 2^m/\omega_m]
\ee
such that
\be \label{eqn-lem-weight-new-key}
\Theta_{*m}(||T||,x)=\theta_T(x)\lambda(x)  
\ee
for $\mathcal{H}^m$ almost every $x\in \set (T)$ and
\be \label{eqn-lem-weight-2}
||T||=\theta_T \lambda \mathcal{H}^m \rstr \set(T),
\ee
where $\omega_m$ denotes the volume of an unitary ball in $\R^m$ and $\theta_T: Z \to \N \cup \{0\}$ is an $L^1$ function called weight given by
\be
\theta_T= \sum_{i=1}^\infty \theta_i\circ\varphi_i^{-1}\One_{\varphi_i\left(A_i\right)}.
\ee
\end{lem}

\begin{lem}[Ambrosio-Kirchheim]\label{lemma-setrect}
If $T \in \intrectcurr_m\left(Z\right)$, then $\set(T)$ is a countably $\mathcal{H}^m$ rectifiable metric space, ie. there exist a countable
collection of bilipschitz charts
\be
\psi_i: A_i \subset \mathbb{R}^n \to U_i \subset Z
\ee
where $A_i$ are Borel measurable sets and 
\be
\mathcal{H}^n(\, \set(T) \, \setminus \,\cup_{i=1}^\infty U_i\,) =0.
\ee
\end{lem}

Finally, we define integral currents. 

\begin{defn}[Ambrosio-Kirchheim]\label{defn-integralcur}
An integral current is an integer rectifiable current,  $ T\in\intrectcurr_m(Z)$, such that $\partial T$ is also a current of finite mass where $\partial T$ is defined by:
\begin{equation}
\partial T \left(f, \pi_1,..., \pi_{m-1}\right) = T \left(1, f, \pi_1,..., \pi_{m-1}\right)
\end{equation}
We denote the space of $m$ dimensional integral currents on $Z$ by $\intcurr_m\left(Z\right)$ .
\end{defn}


\subsection{Integral Current Spaces and SWIF Distance}\label{sbs-IFdistance}

In this subsection we define integral current spaces, Definition \ref{defn-intcurrspace}, the Sormani-Wenger intrinsic flat distance between these spaces, Definition \ref{defn-IFdistance} and state a theorem that shows that the SWIF limit is contained in the GH limit, Theorem \ref{thm-FlatInGH}.

\begin{defn} [Sormani-Wenger]\label{defn-intcurrspace}
Let $\left(Y,d\right)$ be a metric space and $T \in \intcurr_m(\bar{Y})$. If $\set\left(T\right)=Y$ then $(Y,d,T)$ is called an $m$ dimensional integral current space. 
$T$ is called the integral current structure.   $Y$ is called the canonical set.

For technical reasons the zero integral current space is defined. It is denoted by ${\bf{0}}$ and has current $T=0$. 

We denote by $\mathcal{M}^m$ the space of $m$ dimensional integral current spaces and by $\mathcal{M}_0^m$ the space of $m$ dimensional integral current spaces whose canonical set is precompact. 
\end{defn}

Note that we can obtain an integral current space $(M,d,T)$ from a compact oriented Riemannian manifold $(M^n,g)$ (with or without boundary). In this case, $d$ represents the metric induced by $g$ and $T$ is integration over $M$:
\be
T(f, \pi_1,..., \pi_n)=\int_M f d\pi_1 \wedge \cdots \wedge d\pi_n.
\ee

\begin{defn}[Sormani-Wenger]\label{defn-IFdistance}
Let $(Y_i,d_i,T_i) \in \mathcal{M}^m$. Then the intrinsic flat distance between
these two integral current spaces is defined by
\begin{align}\label{eqn-local-defn}
d_{\Fm}\left((Y_1,d_1,T_1),(Y_2,d_2,T_2)\right) 
=\inf\{ d_F^Z(\varphi_{1\#}T_1,\varphi_{2\#}T_2)\}\\
=\inf\{\mass\left(U\right)+\mass\left(V\right)\},
\end{align}
where the infimum is taken over all complete metric spaces,
$\left(Z,d\right)$, and all integral currents,
$U\in\intcurr_m\left(Z\right), V\in\intcurr_{m+1}\left(Z\right)$,
for which there exist isometric embeddings
$
\varphi_i : \left(\bar{Y_i}, d_i\right)\to \left(Z,d\right)
$
with
\begin{equation} \label{eqn-Federer-Flat-2}
{\varphi_1}_\# T_1- {\varphi_2}_\# T_2=U+\bdry V.
\end{equation}
The ${\bf{0}}$ $m$-dimensional integral current isometrically embeds into any $Z$
with $\varphi_\#0=0 \in \intcurr_m\left(Z\right)$. 
\end{defn}

It was proven in Theorem 3.27 of \cite{SorWen2} that $d_{\Fm}$ is a distance on the class of precompact integral current spaces, $\mathcal{M}_0^m$.

We apply the following compactness theorem in all of our SWIF theorems.   It is
proven by Sormani-Wenger applying a combination of Gromov's Compactness Theorem
and Ambrosio-Kirchheim's Compactness Theorem.

\begin{thm}[Sormani-Wenger]\label{thm-FlatInGH}
Let $\left(X_j, d_j, T_j\right)$ be a sequence of $m$ dimensional integral current spaces. If there exist $D$, $M$ and $N:(0,\infty) \to \N$  such that for all $j$
\be
\diam(X_j) \leq D, \,\,\, \mass(T_j) + \mass(\bdry T_j) \leq M
\ee
and, for all $\varepsilon$ there are $N(\varepsilon)$  $\varepsilon$-balls that cover $X_j$, then
\be
\left(X_{j_k}, d_{j_k}\right) \GHto \left(X,d_X\right)\,\,\,\text{and}\,\,
\left(X_{j_k}, d_{j_k}, T_{j_k}\right) \Fto \left(Y,d,T\right),
\ee
where either $\left(Y,d,T\right)$ is an $m$ dimensional integral current space
with $Y \subset X$
or it is the ${\bf 0}$ current space.
\end{thm}

\begin{rmrk}\label{rmrk-FlatInGH}
In a later theorem, Sormani-Wenger constructed a common compact metric space $Z$ and isometric embeddings $\varphi_j: X_j \to Z$ and $\varphi: X \to Z$ such that 
\be
\varphi_j(X_{j_k}) \Hto  \varphi(X) \,\,\text{and}\,\,\varphi_{j_k\#}(T_{j_k}) \Fto \varphi_{\#}(T),
\ee
where $\set(\varphi_{\#}(T)) \subset X$. Here, $\set(\varphi_{\#}(T))$ can be the empty set.   Note that this is not proven using the common compact metric space constructed by Gromov in his work.  In fact, the $Z$ constructed in \cite{SorWen2} is a countably $\mathcal{H}^{n+1}$ rectifiable metric space.
\end{rmrk}

\subsection{Sormani-Wenger: GH=SWIF when there is no boundary}\label{sbs-SWgh=if}

For  a sequence of compact oriented Riemannian manifolds, $(M_j,g_j)$,  with nonnegative Ricci curvature, $\partial M_j=\emptyset$ and two sided uniform volume bounds,  Sormani-Wenger  proved that the GH limit, $X$, of these type of sequences agree with the SWIF limit, Y, \cite{SorWen1}  (cf. Theorem \ref{thm-SWgh=if} below).
In \cite{SorWen1}, Sormani-Wenger prove a far more general theorem about a larger
class of integral current spaces without boundary, and thus the proof is quite technically
complicated.  In this subsection we present an adapted and simplified version of their proof specialized to oriented Riemannian manifolds without boundary based upon Sormani's Geometry Festival presentation of the result.  Moreover, Portegies-Sormani \cite{PorSor}  provides
detailed proofs of some ideas presented there.

\begin{thm}[Sormani-Wenger, Theorem 7.1 in \cite{SorWen1}]\label{thm-SWgh=if}
Let $(M_j,g_j)$ be a sequence of  $n$ dimensional oriented compact Riemannian manifolds with no boundary
that satisfy the following  
\be\label{eq-SWgh=if1}
\ric(M_j) \geq 0, \,\,\, \diam(M_j) \leq D \text{   and  } \vol(M_j) \geq v
\ee
for some constants $v, D >0$.
Then there exist  subsequence and an $n$-integral current space $(X,d,T)$ such that 
\be\label{eq-SWgh=if2}
(M_{j_k},d_{j_k}) \GHto (X,d) 
\ee
and 
\be
(M_{j_k},d_{j_k}, T_{j_k}) \Fto (X,d,T), 
\ee
where $T_j$ is integration of top forms over $M_j$. 
\end{thm}

From (\ref{eq-SWgh=if1}) by Gromov's Compactness theorem \cite{Gromov-metric}, Sormani-Wenger obtain a subsequence converging to a metric
space $(X,d)$ in GH sense, (\ref{eq-SWgh=if2}).  Then applying Sormani-Wenger's theorem \cite{SorWen2} (cf. Theorem~\ref{thm-FlatInGH} above), they get a further subsequence converging in SWIF sense to 
an integral current space $(Y,d_Y,T)$. Note that by Remark \ref{rmrk-FlatInGH} we can suppose that $(M_{j_k},d_{j_k})$, $(X,d)$, $(M_{j_k},d_{j_k}, d_{j_k})$, $(Y,d_Y,T)$ lie in a common metric space. 

By the definition of an integral current space, $x \in Y$ if and only if 
\be 
\liminf_{r \to 0} \frac{  ||T||(B(x,r)) }{\omega_n r^n} > 0.
\ee
Then, to prove that the GH limit coincides with the SWIF limit they estimate 
$||T||(B(x,r))$ for all $x \in X$.

Sormani-Wenger show that for each $x \in \mathcal R(X)$, where $\mathcal R(X)$ denotes the regular points in Cheeger-Colding sense of $X$ (see Definition \ref{defn-regpt}) there is $C(x) >0$ and $r_0(x) >0$ such that 
\be
||T||(B(x,r)) > C(x)r^n \qquad \forall r \leq r_0.
\ee
We now review how they prove this.  Since the mass measure is only lower semicontinuous with respect to SWIF convergence \cite{AK}, they use the notion of filling volume of a current \cite{SorWen1}. The filling volume by definition is smaller than the mass and is  continuous with respect to SWIF convergence:

\begin{defn}\label{defn-fillvol} (c.f. \cite{PorSor})
Given an $n$-integral current space $N=(Y,d,T)$, $n \geq 1$, define  the filling volume of $\bdry N$
by 
\be
\fillvol(\bdry N)= \inf \{ \mass(S) \, | S \text{is an $n+1$ integral current space such that } \bdry S= \bdry N\}.
\ee
That is, there is a current preserving isometry $\varphi: \bdry S \to \bdry N$ such that $\varphi_\sharp \bdry S = \bdry N$.
\end{defn}

Thus, from the definition of filling volume and mass it follows that
\be 
||T||(Y) = \mass(N) \geq \fillvol(\bdry N). 
\ee

The continuity of the filling volume with respect to SWIF convergence follows from the following theorem. This fact was first observed by Sormani-Wenger \cite{SorWen1} building upon work by Wenger on flat convergence of integral currents in metric spaces  \cite{Wenger-flat}.  The precise statement given here is
Theorem 2.48 in work of Portegies-Sormani \cite{PorSor}: 

\begin{thm}[cf. Portegies-Sormani \cite{PorSor}]\label{thm-fillvolCont}
For any pair of integral current spaces, $M_i$, we have
\be
\fillvol(\bdry M_1) \leq \fillvol(\bdry M_2) + d_{\mathcal F}(M_1,M_2).
\ee
\end{thm}

We note that the notion of filling volume given in Definition \ref{defn-fillvol} is not exactly the same notion as the Gromov Filling Volume \cite{Gromov-filling}, however many similar properties hold. Gromov's Filling volume is defined using chains rather than integral current spaces and the notion of volume used by Gromov is not the same as Ambrosio-Kirchheim's mass.   

Now, in order to use the notion of filling volume and its continuity under SWIF convergence to estimate $||T||(B(x,r))$ for $x \in \mathcal R (X)$ we state Portegies-Sormani \cite{PorSor}, Lemma \ref{lem-ICSball}, which allows us to view a ball as an integral current space and Sormani \cite{SorAA}, Theorem \ref{thm-SorArzela}, which allows us to take the limits of balls.

\begin{lem}[cf. Lemma 3.1 in \cite{PorSor}]\label{lem-ICSball}
Let $M$ be a Riemannian manifold, $p \in M$. For almost every $r > 0$, the ball $B(p,r)$
with the current restricted from the current structure of the Riemannian manifold, $T$,  
\be
(\set(T \rstr  B(p,r)), d, T \rstr B(p,r)),
\ee
is an integral current space itself.
\end{lem}

\begin{thm}[Sormani \cite{SorAA}]\label{thm-SorArzela}
Let $\left(X_j, d_j, T_j\right)$ be a sequence of  integral current spaces such that 
\be
\left(X_j, d_j, T_j\right) \Fto \left(X_\infty ,d_\infty,T_\infty \right)
\ee
and $x_j \in X_j$ a Cauchy sequence. 
Then there is a subsequence such that for almost all $r>0$,   
\be
(B(x_j,r) , d_j, T_j \rstr  B(x_j,r) )
\ee
are integral current spaces and 
\be
(B(x_j,r) , d_j, T_j \rstr  B(x_j,r)) \Fto (B(x_\infty,r) , d_\infty, T_\infty \rstr  B(x_\infty , r) ).
\ee
\end{thm}

From Lemma \ref{lem-ICSball} and Theorem \ref{thm-fillvol} and applying
Theorem~\ref{thm-fillvolCont} to a sequence of converging balls we see that
for $x\in \mathcal{R}(X)$ we have
\begin{eqnarray}\label{Txr}
||T||(B(x,r)) &\ge& \fillvol(\bdry (B(x,r) , d, T \rstr  B(x,r)))\\
& =&\lim_{j\to \infty}
\fillvol(\bdry (B(x_j,r) , d_j, T_j \rstr  B(x_j,r))).
\end{eqnarray}
Thus to get a lower estimate of $\fillvol(\bdry (B(x,r) , d, T \rstr  B(x,r)))$ it is enough to estimate
\be\label{SW-estimate}
\fillvol(\bdry (B(x_j,r) , d_j, T_j \rstr  B(x_j,r))) \ge Cr^n \text{ for } x_j \to x. 
\ee

Sormani-Wenger proved the following filling volume estimate in \cite{SorWen1}. It holds in a more general setting and was proven using work of Gromov \cite{Gromov-filling} and Ambrosio-Kirchheim \cite{AK}. In the case of Riemannian manifolds the proof is quite similar to Greene-Petersen's proof of the existence of lower bounds on volume of balls of manifolds with a local geometric contractibility function \cite{Greene-Petersen}.  

\begin{thm}[Sormani-Wenger \cite{SorWen1}, cf. Theorem 3.19 in \cite{PorSor}]\label{thm-fillvol}
Let $(M,g)$ be a compact $n$-dimensional Riemannian manifold (with or without boundary) and $p \in M$. If there exist $r_0 > 0$ and
$k\geq 1$ such that $B(p,kr_0) \cap \set(\bdry M) = \emptyset$, and every $B(z,r) \subset B(p, r_0)$ is contractible within $B(z, kr)$ for all $r \leq 2^{-(n+6)}k^{-(n+1)}r$. Then there is $C_k >0$ such that 
\be 
\fillvol(\bdry (\overline{B(p,r)},d, T \rstr B(p,r))    ) \geq C_k r^n
\ee
for all $r \leq 2^{-(n+6)}k^{-(n+1)}r_0$.
\end{thm}

With the hypotheses of Theorem \ref{thm-SWgh=if}, Sormani-Wenger obtain $r_0$ and $k$ that only depend on $x$ for $x_j$ where $j$ is large.  They use the fact that if $x$ is a regular point of $X$ then all its tangent cones are isometric to $\R^n$ as in Cheeger-Colding \cite{ChCo-PartI} ( cf. Theorem \ref{thm-RegularSet}).  They then apply Colding's Volume Estimate \cite{Colding-survey} (cf. \ref{thm-Colding-volume}) and the GH
convergence of the balls to show that the volume of small balls contained in $B(x_j,r_0)$ satisfy inequality (\ref{eq-eqPerelman}) of Perelman's Main Lemma in \cite{Perelman-max-vol}: 

\begin{thm}[Perelman, Main Lemma and remark \cite{Perelman-max-vol}]\label{thm-Perelman}
For any $c_2 > c_1> 1$ and integer $k >0$ there is $\delta(k,c_1,c_2) > 0$ with the following property:

Let $M$ be an $n$-dimensional Riemannian manifold with $\ric(M) \geq 0$.    
Suppose that  $p \in M$ and that 
\be\label{eq-eqPerelman}
\vol(B(q,\rho)) \ge (1-\delta) \omega_n \rho^n  
\ee
for every ball $B(q,\rho) \subset B(p,c_2R)$. 
Then,
\begin{itemize} 
\item 
any continuous function $f:S^k \to B(p,R)$
can be continuously extended to a function
\be \bar f:B^{k+1} \to B(p,c_1R).
\ee
\item 
any continuous function $f:S^k \to M \setminus B(p,c_1R)$
can be continuously deformed to a function
\be 
\bar f:S^k \to  M \setminus B(p,c_2R).
\ee
\end{itemize}
\end{thm}

This theorem allows them to obtain the contractibility of balls so that they can apply Theorem~\ref{thm-fillvol} to obtain (\ref{SW-estimate}) which implies (\ref{Txr}).  Thus every $x\in \mathcal{R}(X)$ is in the SWIF limit $Y=\set(T)$.

To prove that the non regular points of $X$ are contained in the SWIF limit, Sormani-Wenger use the fact that the singular set of the GH limit has zero measure \cite{ChCo-PartI} (cf. Theorem \ref{thm-RegularSet}) and 
Ambrosio-Kircheim's characterization of the mass measure \cite{AK} (cf. Lemma \ref{lemma-weight}). 




\section{Convergence of Metric Spaces with Boundary}\label{sec-ConvergenceMet}

In this section we prove our new GH compactness theorems for sequences of metric spaces $(X_j,d_j)$:  Theorem \ref{thm-GHconvergence1}, Theorem \ref{thm-GHCollapsing} and Theorem \ref{thm-generalizeConv1}. 
Theorem \ref{thm-GHconvergence1} is applied in all our GH convergence theorems, except Theorem \ref{thm-GHCollapsing}, stated in this paper. The theorem relies on the GH convergence of inner regions $(X^\delta_j,d_j)$ and the boundaries, $(\bdry X_j,d_j)$.  Theorem \ref{thm-GHCollapsing} deals with the collapsed case when the GH limit of $(X_j,d_j)$ agrees with the GH limit of $(\bdry X_j,d_j)$. If the Hausdorff dimension of the GH limit drops then the GH limit cannot agree with the SWIF limit. Under the conditions of Theorem \ref{thm-generalizeConv1} we prove that GH and SWIF limits agree. We provide examples showing the necessity of our hypotheses. 

\subsection{GH convergence: Theorem~\ref{thm-GHconvergence1}}
 
 We first prove the following useful lemma
and then
prove Theorem~\ref{thm-GHconvergence1}.   Recall that if $(X,d)$ is a metric space with non empty boundary, where the boundary is defined as
$\bdry X = \bar X \setminus X$, then $X^\delta= \{ x \in X \,|\, d(x,\bdry X) > \delta \}$.

\begin{lem}\label{lem-cover2cov}
Let $(X,d)$ be a precompact metric space with boundary $\bdry X$ defined as in (\ref{defn-bdry}). If 
$\{B(x_k,\delta)\}_{k=1}^N$ is a cover of $(\bdry X,d)$, then $\{B(x_k,2\delta)\}_{k=1}^N$ is a cover of $( \bar X\setminus X^\delta,d)$.
\end{lem}

\begin{proof}
We have to show that for all $x \in \bar X \setminus X^\delta$ there is $k$ such that $x \in B(x_k,2\delta)$.
Let $x \in \bar X \setminus X^\delta$, then $d(x, \bdry X) \leq \delta$. Since $\bdry X$ is precompact there is $x' \in \overline{\bdry X}$ 
such that 
\be
d(x,x')=d(x, \bdry X) \leq \delta.
\ee
If $\{B(x_k,\delta)\}$ is a cover of $(\bdry X,d)$, then it is a cover of  $(\overline{\bdry X},d)$ and $x' \in \bdry X$ then there is $k$ such that 
\be
d(x',x_k) < \delta.
\ee
 Hence, 
 \be
 d(x,x_k)< 2\delta.
 \ee
This proves that $x \in B(x_k,2\delta)$. The result follows. 
\end{proof}

We now apply this lemma to prove our GH convergence theorem for metric spaces.

\begin{proof}[Proof of Theorem \ref{thm-GHconvergence1}]

In order to prove that $(\bar X_j,d_j)$ converges in  GH sense
we will construct a function 
\be 
N:(0,\infty) \to \N
\ee
for $\{(\bar X_j,d_j)\}$ as in Definition~\ref{defn-equibounded} and then apply Gromov's Compactness Theorem (cf. Theorem~\ref{thm-GromovCompactness}).

Since $\overline{X^{\delta_i}_j}$ and $\overline{\bdry X_j}$ converge in GH sense, by the converse of Gromov Compactness Theorem, cf. Theorem \ref{thm-ConverseGromov}, there exist functions $N(\,\cdot\,,\{\overline{X_j^{\delta_i}}\})$  and  $N( \,\cdot\,,\{\overline{\bdry X_j}\})$, respectively,  that uniformly bound the number of balls needed to cover each element of the sequences. Using these functions we first define $N:\{2\delta_i\} \to \N$. A bound on the number of $2\delta_i$-balls needed to cover $\bar{X_j}$ can be obtained 
by adding the number of  $2\delta_i$-balls needed to cover $\overline{X^{\delta_i}_j}$
to the number of $2\delta_i$-balls needed to cover $\bar X_j \setminus X_j^{\delta_i}$.    With the notation of Definition \ref{defn-equibounded} and applyng Lemma  \ref{lem-cover2cov}, define
\be 
N(2\delta_i) =  N(\delta_i,\left\{
\overline{X_j^{\delta_i}}\right\})  + N(\delta_i,\big\{\overline{\bdry X_j}\big\}).
\ee

The domain of $N$ is extended to $(0,\infty)$ by defining 
\be
N(\varepsilon)=N(2\delta_i)
\textrm{ where }2\delta_{i+1} \leq \varepsilon < 2\delta_i
\ee
and $N(\varepsilon)=N(2\delta_1)$ for $\varepsilon > 2\delta_1$.

Since $X_j$ is a length metric space and can be covered with $N(\delta_1)$ balls with radius $\delta_1$ then 
\be
\diam(\bar{X_j})\leq 2\delta_1N(\delta_1).
\ee
 The result follows from Gromov's Compactness Theorem (cf. Theorem~\ref{thm-GromovCompactness}). 
\end{proof}

We now present an example demonstrating that the conclusion of Theorem~\ref{thm-GHconvergence1} does not hold if the sequence of boundaries does not converge. 

\begin{figure}[h] 
 \centering 
\includegraphics[width=3.5in]{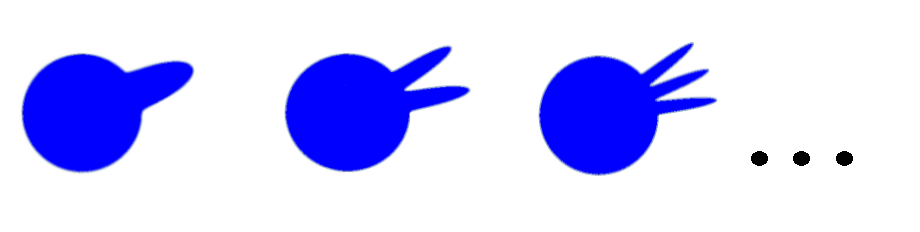} 
\caption{$\bar X_1$, $\bar X_2$, $\bar X_3$...}
\label{fig1}
\end{figure}

\begin{ex}\label{ex-multipleSplines}
Let $X_j$ be contained in $\R^n$, $n\ge 2$, endowed with the length metric that comes from the standard metric defined on $\R^n$. See Figure~\ref{fig1} above. Each $X_j$ consists of an open ball of radius $r$ with $j$ increasingly thin splines of constant length, $L$ and width $w_j \to 0$.  

Observe that each $X_j$ is precompact and that its boundary, $\bdry X_j$, is compact. For each $\delta >0$, there is  $N$ such that $X_j^\delta$ does not contain any spline for $j \geq N$.  Actually, the sequence $(\overline{X_j^\delta},d_j)$ converges in GH sense to a closed ball,
\be
(\overline{X_j^\delta},d_j) \,\,\GHto\,\, 
\overline{B(0, r-\delta)}.
\ee

Due to the increasing number of splines of constant length, the sequence $\{(\bdry X_j,d_j)\}_{j=1}^{\infty}$ is not equibounded. Thus, by the converse of Gromov Compactness Theorem, (cf. Theorem \ref{thm-ConverseGromov}),  $\{(\bdry X_j,d_j)\}_{j=1}^{\infty}$  does not have any GH convergent subsequences. 
Note that $\{(X_j,d_j)\}_{j=1}^{\infty}$ does not have GH convergent subsequences for the same reason. 
\end{ex}

The next example is modeled after the pictures depicted in Frank Morgan's book \cite{Morgan-text}.  It shows that in  Theorem~\ref{thm-GHconvergence1} the GH convergence of sequences of $\delta_i$-inner regions is necessary.  

\begin{figure}[h] 
 \centering
\includegraphics[width=3.5in]{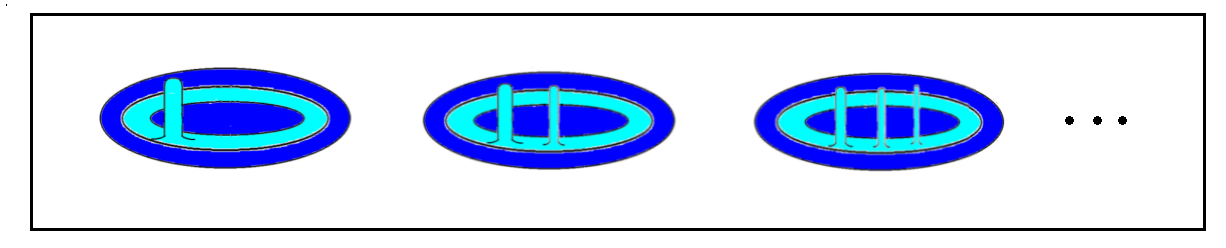} 
\caption{$X_1$, $X_2$, $X_3$...}
\label{fig2}
\end{figure}

\begin{ex}\label{ex-multipleSplines2}
Let $0<c< C <1$ be constant numbers. Consider the sequence $(X_j,d_j)$ of 
precompact surfaces in 3-dimensional Euclidean space, $\mathbb{E}^3$
as depicted in Figure~\ref{fig2} above.  More precisely, 
\be
X_j= \{ (x,y,0) \in \R^3| x^2 + y^2 \leq c\} \cup A_j \cup \{ (x,y,0) \in \R^3| C \leq x^2 + y^2 < 1\},
\ee
where $A_j$ (depicted in light blue in Figure~\ref{fig2}) has $j$ increasingly thin splines of constant height located in such a way that for all $(x,y,0) \in \bdry X_j$: 
\be \label{eq-metricSplines2}
d_{j}((x,y,0), (0,0,0)) \leq 1-c + \alpha_j,  
\ee
where $\alpha_j \to 0$ and the $d_j$'s are length metrics induced by the Euclidean metric $d_{\mathbb{E}^3}$. 

By the condition on the metrics, Equation (\ref{eq-metricSplines2}), it follows that 
\be 
(\bdry X_j,d_j) \GHto (\{ (x,y,0) \in \R^3| x^2 + y^2 = 1\},d_{\mathbb{E}^3}). 
\ee 
For $\delta \in [0, 1-c)$, the sequence $\{(\overline{X_j^\delta},d_j)\}_{j=1}^\infty$ is not equibounded due to the splines in $A_j$. Then, by the converse of Gromov's compactness theorem (cf. Theorem \ref{thm-ConverseGromov}),  $\{(\overline{X_j^\delta},d_j)\}_{j=1}^\infty$ does not have any GH convergent subsequence, nor the sequence $\{\bar X_j\}$. 
\end{ex}

\subsection{Collapsing to the Boundary or Not}

In this subsection we prove the following collapsing theorem, Theorem \ref{thm-GHCollapsing}. Then we describe a sequence of cylinders that collapses to a segment, Example \ref{ex-collapsingCylinders}.

\begin{thm}\label{thm-GHCollapsing}
Let $(X_j,d_j)$ be a sequence of precompact length metric spaces with compact boundary.  Suppose that there is a compact metric space 
$(X_\bdry, d_\bdry)$ such that
\be 
(\bdry X_j, d_j) \GHto (X_\bdry, d_\bdry).
\ee 
Then either: 
\begin{enumerate}
\item there is $\delta >0$ such that $X_j^\delta \neq \emptyset$ for infinitely many j or 
\item 
$(\bar X_j,d_j) \GHto (X_\bdry,d_\bdry)$.
\end{enumerate}
\end{thm}

\begin{rmrk}
When the sequence of boundaries converge and (1) in Theorem \ref{thm-GHCollapsing} is satisfied we cannot conclude anything. Example \ref{ex-multipleSplines2} shows a sequence that satisfies these two conditions but $(X_j,d_j)$ does not have a GH convergent subsequence. Meanwhile, the sequence from Example \ref{ex-oneSpline} satisfies both conditions and GH converges. 
\end{rmrk}

In the next example we describe a sequence of length metric spaces that illustrates the case in which the GH limit of $X_j$ equals the 
GH limit of $\bdry X_j$.

\begin{ex}\label{ex-collapsingCylinders}
Let $X_j$ be a sequence of increasingly thin cylinders in $\R^n$, 
\be 
X_j=\{(x_1,...,x_n) \in \R^n \,| \,x_1^2+\cdots+ x_{n-1}^2 < r_j,\,\,-1 < x_n < 1\},
\ee
$r_j=1/j$, with the restricted standard metric of $\R^n$.
With this metric each $X_j$ is a precompact length metric space and it is clear that each $\bdry X_j$ is non empty and  precompact.
Let 
\be 
X_\infty=\{0\}\times\{0\}\times \cdots \times \{0\}\times [-1,1] \subset \mathbb R^n
\ee

Then,  
\be
d_H^{\R^n}(\bar X_j,X_\infty)< 2/j. 
\ee
Also, 
\be
d_H^{\R^n}(\bdry X_j,X_\infty)< 2/j. 
\ee
Thus, 
\be
\bdry X_j \GHto X_\infty \text{   and  } \bar X_j \GHto X_\infty. 
\ee
\end{ex}

\begin{proof}[Proof of Theorem \ref{thm-GHCollapsing}]
Suppose that there is no $\delta >0$ such that $X_j^\delta \neq \emptyset$ for infinitely many $j$. Fix $\delta>0$, then $X_j^\delta = \emptyset$ for finitely many $j$. Thus, for except a finite number of $j$'s,
$\bar X_j=\bar X_j \setminus X_j^{\delta}$. Hence,  we define a function $N$ that counts the number of $\varepsilon$-balls needed to cover $\bar X_j$ by
\be 
N(2\delta) := N(\delta,\{\bdry X_j\}),
\ee 
where we are using the notation of Definition \ref{defn-equibounded} and applying Lemma  \ref{lem-cover2cov} 

Since each $\bar X_j$ is a length space and can be covered by $N(\varepsilon_0)$ $\varepsilon_0$-balls we get $\diam(\bar{X_j})\leq 2\varepsilon_0 N(\varepsilon_0)$. The result follows from Gromov's compactness theorem, Theorem \ref{thm-GromovCompactness}. 
\end{proof}

\subsection{GH=SWIF: Theorem \ref{thm-generalizeConv1}}

In this subsection we prove Theorem \ref{thm-generalizeConv1}. Theorem \ref{thm-generalizeConv1} assures that the GH and SWIF limits agree for sequences of integral currents that converge in GH and SWIF sense that satisfy conditions (\ref{inside-inside}) and  (\ref{eq-completionY}), namely:
\be
X(\delta_i) \subset Y  \,\, \text{and}\,\, X \setminus X_\bdry \subset Y.
\ee

 In Example \ref{ex-oneSplineInside} we present a sequence that has GH and SWIF limits, that satisfies (\ref{eq-completionY}) but does not satisfy (\ref{inside-inside}). Then, in Example \ref{ex-oneSpline} we describe a sequence that has GH and SWIF limits, that satisfies  (\ref{inside-inside}) but does not satisfy (\ref{eq-completionY}). In both cases, the conclusion of the theorem does not hold.
At the end of the subsection we prove Lemma \ref{lem-IntPtX} that characterizes the points in $X \setminus X_\bdry$. 
We use this lemma to prove Theorem \ref{thm-GH=IF}. 

\begin{proof}[Proof of Theorem \ref{thm-generalizeConv1}]
By hypothesis we know that the SWIF limit is contained in the GH limit, $Y \subset X$. We only have to show that $X \subset \bar Y$.  

Let $x \in X$. Since the sequence $\{\bar X_j\}$ converges in GH sense to $X$, there is a sequence $x_j \in \bar X_j$ that converges to $x$. If there is $i>0$ such that $x_j \in X_j^{\delta_i}$ for infinite many $j$, then by the GH convergence of the sequence $\{X_j^{\delta_i}\}$,  $x \in X(\delta_i)$. Thus, by hypothesis (\ref{inside-inside}), $x \in Y$. 

Otherwise, for each $i$ there is $j_k$
such that $x_{j_k} \notin X_{j_k}^{\delta_i}$. Thus, $d(x_{j_k},\bdry X_{j_k}) \leq \delta_i$. Since each boundary, $\bdry X_j$, is precompact we can choose $y_{j_k} \in \overline{\partial X_{j_k}}$ such that 
\be 
d_{j_k}(x_{j_k},y_{j_k})=d(x_{j_k},\overline{\partial X_{j_k}}) \leq \delta_i.
\ee 
Then by  the triangle inequality, 
\be 
d_{j_k}(y_{j_k},x) \to 0\,\, { as }\,\, k \to \infty.
\ee
Thus, $x \in X_\bdry$. From hypothesis (\ref{eq-completionY}) follows that $x \in \bar Y$. Hence, $X \subset \bar{Y}$. This together with $ Y \subset X$ implies that $X=\bar Y$.
\end{proof}

\begin{rmrk}
From Example \ref{ex-multipleSplines2} it follows that in Theorem \ref{thm-GHconvergence1} the convergence of the sequences of inner regions is necessary.
\end{rmrk}

In the next example we describe a sequence that satisfies all the conditions of Theorem \ref{thm-GHconvergence1}, except condition (\ref{inside-inside}). The conclusion of the theorem does not hold. 

\begin{figure}[h] 
   \centering
   \includegraphics[width=3.5in]{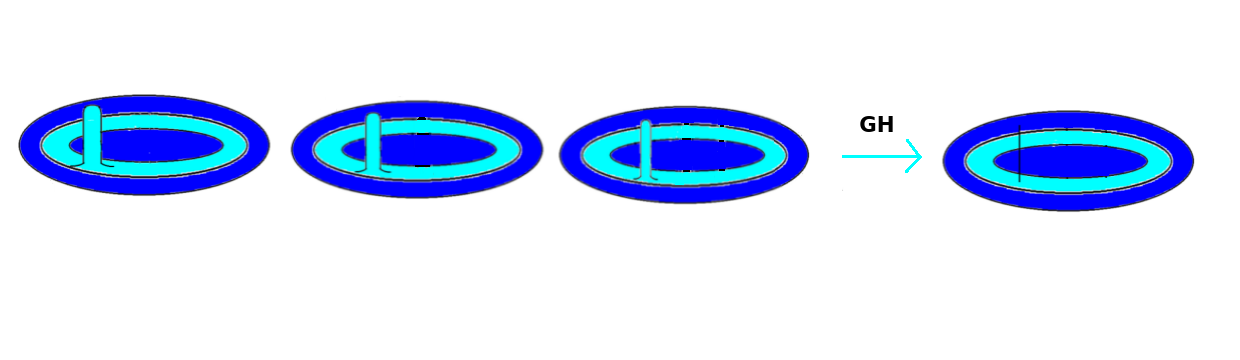} 
 \caption{$X_1$, $X_2$, $X_3$...}
\end{figure}

\begin{ex}\label{ex-oneSplineInside}
Just as in Example \ref{ex-multipleSplines2}, let $0<c< C <1$ be constant numbers. Consider the sequence $(X_j,d_j)$ of 2-dimensional precompact length metric spaces in 3-dimensional Euclidean space as depicted in the figure above given by: 
\be
X_j= \{ (x,y,0) \in \R^3| x^2 + y^2 \leq c\} \cup A_j \cup \{ (x,y,0) \in \R^3| C \leq x^2 + y^2 \leq 1\},
\ee
where $A_j$ has only one increasingly thin spline of constant height located in such a way that  
\be 
d_j((x,y,0), (0,0,0)) \leq 1-c + \alpha_j  
\ee
for all $(x,y,0) \in \bdry X_j$ and $\alpha_j \to 0$. 

The spline in $A_j$ converges to a segment. Thus, the GH limit of $(\bar X_j,d_j)$ is a disc with the segment attached to it 
and its SWIF limit, $Y$, equals the disc. 

The sequences $(\overline{X_j^\delta},d_j)$ GH converge for all $\delta <1$: 
\be 
(\overline{X_j^\delta},d_j) \GHto (X(\delta), d_\delta).
\ee 
But $(X(\delta), d_\delta) \nsubset \bar Y$ for $\delta \in [0,1-c)$ since each $X(\delta)$ contains a segment that $Y$ does not. Hence, this sequence does not satisfy condition (\ref{inside-inside}). It satisfies $X_\bdry \subset Y$ since 
from Equation (\ref{eq-metricSplines2}) follows that 
\be 
(\bdry X_j,d_j) \GHto (\{ (x,y,0) \in \R^3| x^2 + y^2 = 1\},d_{\mathbb{E}^3}). 
\ee 
\end{ex}



In the next example we show a sequence that satisfies all the conditions of 
Theorem \ref{thm-generalizeConv1}, except $X_{\partial} \subset Y$, for which the conclusion of the theorem does not hold.  

\begin{figure}[h] 
 \centering
\includegraphics[width=3.5in]{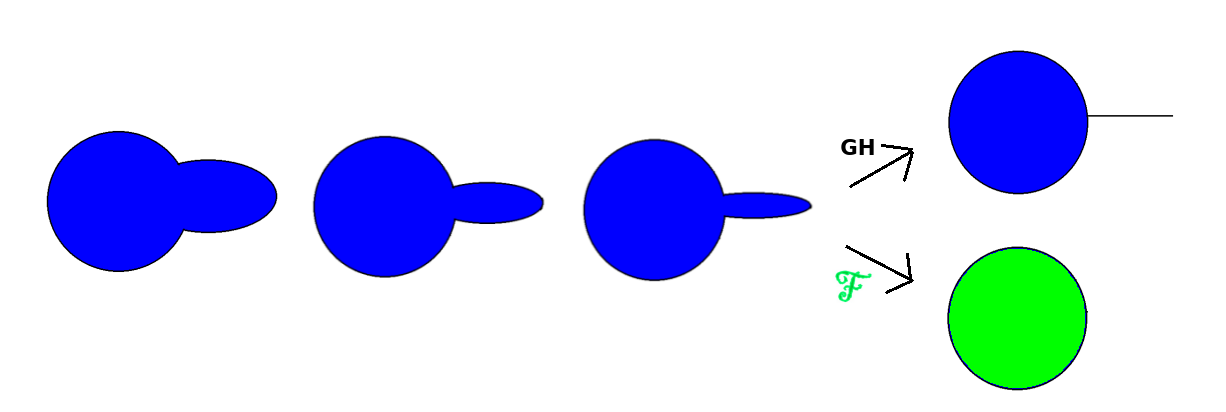} 
\caption{$X_1$, $X_2$, $X_3$...}
\label{fig-OneSpline}
\end{figure}

\begin{ex}\label{ex-oneSpline}   Let $M_j$ be a sequence of $n$-dimensional manifolds diffeomorphic to a closed ball in $\mathbb R^n$, $n \geq 2$,
consisting of a ball with a single increasingly thin spline as depicted in Figure \ref{fig-OneSpline} for $n=2$. Let $X_j = M_j \setminus \bdry X_j$.  Then, each $X_j$ is a precompact length metric space with compact boundary.  

The GH limit, $X$, of the sequence is a closed ball of radius $r$, $\overline{B(0,r)}$, with a segment attached.  The SWIF limit of the manifolds is  $Y=\overline{B(0,r)}$. 

For each $\delta >0$, there is  $N$ such that $X_j^\delta$ does not contain any spline for $j \geq N$.  Actually, the sequence $(\overline{X_j^\delta},d_j)$ converges in GH sense to a closed ball,
\be
(\overline{X_j^\delta},d_j) \,\,\GHto\,\, 
\overline{B(0, r-\delta)}.
\ee
Hence, $\overline{B(0, r-\delta)} \subset Y$. The GH limit of the boundaries is a sphere with a segment attached.  Thus, $X_\partial \nsubset Y$. 
\end{ex}

Now we characterize the points in $X \setminus X_\bdry$. This characterization will help us to prove GH=SWIF 
in Section \ref{sec-ConvergenceMan}.

\begin{lem}\label{lem-IntPtX}
Let $\{(X_j, d_j)\}_{j=1}^\infty$ be a sequence of precompact metric spaces with compact boundary
that converges in GH sense to a compact metric space $(X,d)$:
\be
(\bar X_j, d_j) \GHto (X, d).
\ee
Denote by $X_\bdry$ the GH limit of the boundaries:
\be 
(\bdry X_j, d_j) \GHto (X_\bdry, d).
\ee
Let $x \in X$. Then $x \in X \setminus X_\bdry$ if and only if there is $\delta >0$ and a sequence $x_j \in X_j^\delta$ such that 
\be
\lim_{j \to \infty} x_j=x. 
\ee
\end{lem}

\begin{proof}

Suppose that there is a sequence  $x_j \in X_j^\delta$ that converges to $x$ and that $x$ is contained in 
$X_\bdry$. Then, by the GH convergence of the boundaries, there is a sequence $y_j \in \bdry X_j$ that converges to $x$ as well. By the triangle inequality:
\be
d_j(x_j,y_j) \leq d_j(x_j, x)  + d_j(x,y_j). 
\ee
Since $d_j(x_j, x),d_j(x,y_j) \to 0$ as $j \to \infty$, for $j$ big enough, we have 
\be 
d_j(x_j,y_j) < \delta,
\ee
which is a contradiction since $x_j \in X_j^\delta$.

Now, suppose that there is no $\delta$ and no sequence as in the statement of the lemma. Since $x \in X$, there is  $x_j \in M_j$ that converges to x. By our supposition, for each $k$ it is possible to choose $k \in \N$
such that 
\be
 d(x_{j_k}, \bdry X_{j_k}) \leq 1/k.
\ee
Since $\bdry X_{j_k}$ is compact there exists $y_{j_k} \in \bdry X_{j_k}$ such that $d(x_{j_k}, \bdry X_{j_k})=d(x_{j_k}, y_{j_k})$. 
Hence, there is a subsequence $\{ y_{j_k}\} \subset \bdry X_{j_k}$ that converges to $x$. Hence,  $x \in X_\bdry$.

\end{proof}



\section{Convergence of Riemannian Manifolds with Boundary}\label{sec-ConvergenceMan}

In this section we prove the compactness theorems for manifolds stated in the introduction: Theorem \ref{thm-GHconvergence2}, Theorem \ref{thm-GH=IF} and Theorem \ref{thm-injectivity}. These theorems are consequence of the general cases that hold for metric spaces: Theorem \ref{thm-GHconvergence1} and Theorem \ref{thm-generalizeConv1}. Theorem \ref{thm-GHconvergence2} about GH convergence is proven in Subsection \ref{sbs-GHconvergence2}. Then in Subsection \ref{sbs-GH=IF} we show that the GH limit is contained in the SWIF limit, Theorem \ref{thm-GH=IF}. This is done adapting Sormani-Wenger's GH=SWIF theorem for manifolds with no boundary \cite{SorWen1}. In Subsection \ref{sbs-injectivity} we prove that the GH limit agrees with the completion of the SWIF limit when having a uniform lower bound on the boundary injectivity radii, Theorem \ref{thm-injectivity}.  To do so we first obtain diameter bounds of $\delta$-inner regions, Lemma \ref{lem-diameters}, and then prove that the sequence of boundaries converges in GH sense, Lemma \ref{lem-bdryInjconv}.

\subsection{GH Convergence: Theorem~\ref{thm-GHconvergence2}} \label{sbs-GHconvergence2}
In this subsection we prove a GH convergence theorem for manifolds with boundary, Theorem \ref{thm-GHconvergence2}.


\begin{proof}[Proof of Theorem \ref{thm-GHconvergence2}]
By hypotheses, we have a sequence of manifolds with nonnegative
Ricci curvature, $M_j$, satisfying (\ref{eq-PerSorClass1})-(\ref{eq-PerSorClass2}).  
These hypotheses are the same as the ones in prior work of the author with Sormani  \cite{PerSor-2013} (cf. Theorem~\ref{thm-GHdeltas}). Thus, applying that theorem we obtain a subsequence $\{j_k\}$ such that the sequence of inner regions converges for all $i$:
\be
(\bar{M}^{\delta_i}_{j_k}, d_{M^{\delta_i}_{j_k}}) \GHto (X(\delta_i),d_{X(\delta_i)}).
\ee

By hypothesis $(\partial M_j, d_{M_j})$ also converges in GH sense.   Thus the hypotheses of our GH convergence theorem for metric spaces, Theorem \ref{thm-GHconvergence1},  proven in Section \ref{sec-ConvergenceMet} above, are satisfied. By applying  Theorem \ref{thm-GHconvergence1} we  
obtain a further subsequence, also denoted $j_k$, such that
the sequence 
 $\{(M_{j_k}, d_{M_{j_k}})\}$  converges in the Gromov-Hausdorff sense.
\end{proof}

The following remark shows a sequence that satisfies all the conditions of Theorem~\ref{thm-GHconvergence2}. So the sequence has a GH convergent subsequence.

\begin{rmrk}\label{rmrk-OneSpline}
In Example \ref{ex-oneSpline} we describe a sequence $M^n_j$ of $n$-dimensional compact Riemannian manifolds with boundary that satisfy all the conditions of Theorem \ref{thm-GHconvergence2}:
\be
\ric(M_j)=0,\,\, \vol(M_j)\leq \omega_n(r+L)^n, \,\, \diam({M^\delta_j})\leq 2(r+L)
\ee
and the center $q_j$ of the closed ball to which the spline is attached in each $M_j$ satisfies
\be 
\vol(B(q_j,\delta))=\omega_n\delta^n,
\ee
where $\omega_n$ is the volume of a ball in $\R^n$ of radius $1$. 
Finally, the sequence of boundaries, $\bdry M_j$, also converges in GH sense. 
\end{rmrk}

%
%
%
%
%

\begin{rmrk}
In Theorem~\ref{thm-bdryConv} we will prove that the assumption on the 
GH convergence of the boundaries can be replaced 
by an assumption on $\diam(\bdry M_j,d_{\bdry M_j})$ and the second fundamental form of the boundaries, $\bdry M_j$, and their derivatives, (\ref{eq-bdryConv}). 
\end{rmrk}

\subsection{Proving GH=SWIF Theorem \ref{thm-GH=IF} and Examples}\label{sbs-GH=IF}

In this subsection we prove Theorem \ref{thm-GH=IF}. The key step is to show that $X \setminus X_\bdry \subset Y$. A sequence of compact, oriented, $n$-dimensional Riemannian manifolds with nonnegative Ricci curvature and positive uniform upper and lower bounds on the diameter and volume, respectively, have subsequences that converge in GH and SWIF sense, cf. Theorem \ref{thm-GromovRic} and Wenger's Compactness Theorem \cite{Wenger-compactness}. Sormani-Wenger proved that when the manifolds have no boundary a single subsequence can be chosen in such a way that both limits coincide \cite{SorWen1}, cf. Theorem \ref{thm-SWgh=if}. By avoiding the boundaries of the manifolds and provided with uniform lower bounds on the volume of the balls, Equation ( \ref{eq-noncollapsing}), Sormani-Wenger's proof can be adapted to show that $X \setminus X_\bdry \subset Y$.   

In Example \ref{ex-bdryGHandIF} and Example \ref{ex-bdryGH=IF} we describe sequences of manifolds that satisfy all the conditions stated in Theorem \ref{thm-GH=IF}. What is interesting to see is that the GH limit and SWIF limit of $\{\bdry M_j\}$ do not need to coincide in order to get that the GH and the SWIF limits of $\{M_j\}$ agree.

\begin{proof}[Proof of Theorem \ref{thm-GH=IF}]
By Theorem \ref{thm-GHconvergence2} there exist compact metric spaces $(X(\delta_i), d_{X(\delta_i)})$ and $(X,d_X)$ such that
\be  \label{eq-GH}
(M_{j_k},d_{j_k}) \GHto (X,d_X)
\ee
and 
\be  \label{eq-GH}
(M^{\delta_i}_{j_k},d_{j_k}) \GHto (X(\delta_i), d_{X(\delta_i)}).
\ee

By  Sormani-Wenger \cite{SorWen2},  cf. Theorem \ref{thm-FlatInGH} , there exist an $n$-integral current space $(Y ,d_Y,T)$ and a further subsequence that we denote in the same way such that 
\be \label{eq-IF}
(M_{j_k},d_{j_k}, T_{j_k}) \Fto (Y,d_Y,T) 
\ee
where either $(Y,d_Y,T)$ is the zero integral current space or $Y \subset X$ and $d_Y=d_X|_Y$. With no loss of generality our convergent subsequences will be indexed by $\{j\}$.

Cheeger and Colding classified the points of a GH limit space into regular and nonregular according to their tangent cones \cite{ChCo-PartI},  cf. Definition \ref{defn-regpt}. Based on this, we divide the proof of the theorem in the following three claims. 

Recall that from the definition of integral current space, Definition \ref{defn-intcurrspace}, $x \in Y$ if and only if
\be\label{set-T}
\liminf_{r\to 0} \frac{\|T\|(B(x,r))}{\omega_n r^n} >0
\ee
holds.

{\bf{Claim 1:}} {\textit{ Let $x  \in   \mathcal R(X(\delta_i))$ and 
$x_j \in M^{\delta_i}_j$  such that $x_j \to x$. Then there is $r_0(x) >0$ and $k(x) \geq 1$
such that $B(x_j,r_0) \cap \set(\bdry T_j) = \emptyset$ and every $B(z,r) \subset B(x_j, r_0(x))$ is contractible within $B(z, kr)$.}}

{\bf{Proof of Claim 1:}} Since $x$ is a regular point contained in $X(\delta_i)$ by Cheeger-Colding \cite{ChCo-PartI} ( cf. Theorem \ref{thm-RegularSet} and Remark \ref{rmrk-singSet}),  every tangent cone of $x$ is isometric to $\R^n$. Thus,  for any $\alpha > 0$ there is $r(\alpha) \leq \delta_i$ such that 
\be
d_{GH}(B(x,r), B(0,r)) < \alpha r/2,
\ee
where $B(0,r) \subset \R^n$ denotes the Euclidean ball in $\R^n$ with radius $r$ and center $0$. Now $x_j  \to x$ and  $M_j \GHto X$ imply that for large $j$
\be
d_{GH}(B(x_j,r), B(x,r)) < \alpha r/2. 
\ee
Using the triangle inequality we obtain that
\be\label{eq-xinyGH}
d_{GH}(B(x_j,r), B(0,r)) < \alpha r
\ee
for large j. 

For $\varepsilon >0$, by Colding's Volume Convergence Theorem \cite{Colding-volume} (cf. Theorem \ref{thm-Colding-volume}) there is $\alpha(\varepsilon,n) >0$ such that 
\be\label{eq-ColdingVol}
\vol(B(x_j,r)) \geq (1 - \varepsilon) \omega_n r^n
\ee
holds if (\ref{eq-xinyGH}) is satisfied. By what we proved in the previous paragraph (\ref{eq-ColdingVol}) holds for large $j$ taking $\alpha= \alpha(\varepsilon, n)$ and $r(\alpha) < \delta_i$. 

Now, for any $k$, by Perelman's Contractibility Theorem 
\cite{Perelman-max-vol} ( cf. Theorem \ref{thm-Perelman} within),  there is $\varepsilon$ such that if (\ref{eq-ColdingVol}) holds then each $B(z,r) \subset B(x_j,r_0(x))$ is contractible within $B(z,kr)$. This finishes the proof of Claim 1.

{\bf{Claim 2:}} {\textit{$\mathcal  R(X(\delta_i)) \subset Y$. That is, (\ref{set-T}) holds for all the regular points of $X(\delta_i)$.}}

{\bf{Proof of Claim 2:}}  Since the result of Claim 1 holds then by Sormani-Wenger Filling Volume Theorem \cite{SorWen1} which uses work by Gromov \cite{Gromov-filling} and was also proven in \cite{Greene-Petersen}, cf. Theorem \ref{thm-fillvol}, there exists $C_{k(x)}>0$ such that 
\be 
\fillvol(\bdry (\overline{B(x_j,r)},d_j, T_j \rstr B(x_j,r))) \geq C_{k(x)} r^n
\ee
for sufficiently large $j$ and all $r \leq 2^{-(n+6)}k^{-(n+1)}r_0(x)$. 

Thus, by the continuity of the filling volume under SWIF convergence \cite{PorSor}, cf. Theorem \ref{thm-fillvolCont}, we have 
\be
\fillvol(\bdry (\overline{B(x,r)},d, T \rstr B(x,r))) 
\geq C_{k(x)} r^n
\ee
for all $r \leq 2^{-(n+6)}k^{-(n+1)}r_0(x)$. Since 
\be 
||T||(S) \geq \fillvol(\bdry S)
\ee 
holds for all integral current spaces $S$. Then, 
\be
||T||(B(x,r)) \geq C_{k(x)} r^n.
\ee
Thus, (\ref{set-T}) holds which proves that $x \in Y$. Hence, $\mathcal R(X(\delta_i)) \subset Y$ and $(Y,d,T)$ is not the zero integral current space.

\item{ {\bf{Claim 3:}} {\textit{If $\mathcal R(X(\delta_i)) \subset Y$ for all $\delta_i \leq \delta$, then $\mathcal  S(X(\delta_i)) \subset Y$, ie. (\ref{set-T}) holds for all the nonregular points of $X(\delta_i)$.}}} 

{\bf{Proof of Claim 3:}} Let $x \in \mathcal S(X(\delta_i))$. From the characterization of the mass measure given by  Ambrosio-Kirchheim \cite{AK} (cf. Lemma \ref{lemma-weight}),  for $r >0$
\be\label{eq-AKMassball}
 ||T||(B(x,r))  \geq \lambda_n  \mathcal H ^n(B(x,r) \cap Y), 
\ee
where $\lambda_n > 0$ only depends on $n$.

Now we bound $\mathcal H ^n(B(x,r) \cap Y)$. First we note that 
\be\label{eq-HausdorffYtoX}
\mathcal H ^n(B(x,r) \cap Y) \geq \mathcal H ^n(B(x,r) \cap X(\delta_k))
\ee
since in Claim 2 we proved that $\mathcal R(X(\delta_k)) \subset Y$ for all $k$ and by Cheeger-Colding \cite{ChCo-PartI} (cf. Theorem \ref{thm-RegularSet} and  Remark \ref{rmrk-singSet}) $\mathcal H^n(\mathcal R(X(\delta_k))=\mathcal H^n(X(\delta_k))$. 

Now we estimate $\mathcal H ^n(B(x,r) \cap X(\delta_k))$. 
Actually, if $r \leq \delta_i - \delta_{i+1}$ we have 
\be
B(x,r) \subset X(\delta_{i+1}).
\ee
Thus, we only need to estimate $\mathcal H ^n(B(x,r))$. 
By the volume estimate calculated by the author and Sormani \cite{PerSor-2013}, cf. Remark \ref{rmrk-VolBall},  
\be 
\vol(B(x_j,r)) \geq v(\delta_i)r^n,
\ee
where $x_j \in M_j^{\delta_i} \to x$ and $r \leq \delta_i/2$. Applying Colding's Volume Convergence \cite{Colding-volume} (cf. Therem \ref{thm-Colding-volume})
we have
\be\label{eq-Hausdorffball}
\mathcal H ^n(B(x,r))\geq  v(\delta_i) r^n
\ee
for $r \leq \delta_i/2$.

Thus, for $r<\min\{\delta_i/2, \delta_i-\delta_{i+1}\}$ from  (\ref{eq-AKMassball}), (\ref{eq-HausdorffYtoX}) and (\ref{eq-Hausdorffball}), 
we see that 
\be
||T||(B(x,r)) \geq \lambda_n  \mathcal H ^n(B(x,r) \cap Y) \geq \lambda_n  v(\delta_i) r^n.
\ee
From this, Equation (\ref{set-T}) holds. Thus, $x \in Y$. This proves $S(X(\delta_i)) \subset Y$.

From Claim 2 and Claim 3 we conclude that $X(\delta_i) \subset Y$ for all $i$. By hypotheses we know that $X_\bdry \subset Y$. Then we get by Theorem \ref{thm-GHconvergence1} that $X=Y$ and by Lemma \ref{lem-IntPtX} that $X \setminus X_\bdry \subset Y$.  
\end{proof}

In Example \ref{ex-oneSpline} we described a sequence that satisfies all the conditions of Theorem \ref{thm-GH=IF}, except $X_{\partial} \subset Y$. See Remark \ref{rmrk-OneSpline}. There the conclusion of the theorem does not hold.  


From Gromov's compactness theorem, cf. Theorem \ref{thm-GromovCompactness}, we know that if $(M_j,d_j)$ converges in GH sense then a further subsequence of $(\bdry M_j,d_j)$ also converges in GH sense. As in Federer-Fleming, if $(M_j,d_j,T_j)$ converges in SWIF sense then $(\bdry M_j,d_j, \bdry T_j)$ converges in SWIF sense \cite{FF} \cite{AK} \cite{SorWen2}.

In the following two examples we present sequences $\{M_j\}$
of manifolds that satisfy all the conditions of Theorem \ref{thm-GH=IF}. 
But for which the  GH limit and SWIF limits of  $\{\bdry M_j\}$ do not agree.
Hence, the hypothesis $X_\bdry \subset Y$ in Theorem \ref{thm-GH=IF} cannot be replaced to requiring that both limits of $\{\bdry M_j\}$ coincide.

\begin{ex}\label{ex-bdryGHandIF}
Let $M_j= S^n \setminus B(p,1/j)$, $n \geq 2$, be the standard $n$-dimensional unit sphere minus a ball of radius $1/j$ with center the north pole, $p$. Each $M^n_j$ is a compact oriented Riemannian manifold with boundary that satisfies 
\be
\ric(M_j)=n-1, \vol(M_j)\leq \vol(S^n), \diam({M^\delta_j})\leq \diam(S^{n-1})
\ee
and the south pole $q \in S^n$ satisfies
\be 
\vol(B(q,\delta))=\frac{\vol(S^n)}{\pi^n} \delta^n
\ee
for small $\delta$. 

The SWIF limit of $M_j$ is $S^n$ and the north pole is the GH limit of $\bdry M_j$. Thus, $\{p\} \subset Y$. These shows that this example satisfies all the hypotheses of Theorem \ref{thm-GH=IF}. Hence, the GH limit of $M_j$ is $S^n$.  But the SWIF limit of $\bdry M_j$ is the zero current. This 
example shows that both limits of $M_j$ can agree even though the limits of $\bdry M_j$ do not.   
\end{ex}

\begin{ex}\label{ex-bdryGH=IF}
Let 
\be
I^3=[-1,1]\times[-1,1]\times[-1,1] \subset \R^3.
\ee
For $j\geq 2$, let
\be 
M_j=I^3 \setminus (-1/j,1/j) \times (-1/j,1/j) \times (0,1] 
\ee 
with the induced flat metric. Notice that the elements of this sequence are not manifolds but they can be smoothened. 
This sequence converges in both GH and SWIF sense to the cube 
\be
X=Y=I^3.
\ee
The boundaries, however, have different limits.
\be
\bdry M_j \GHto \bdry  I^3\,\, \cup\,\, \{0\}\times \{0\} \times [0,1]
\ee
However the SWIF limit of boundaries is the boundary
of the limits:
\be
\bdry M_j \Fto \bdry I^3.
\ee
This shows that both limits of $M_j$ can agree even though the limits of $ \bdry M_j$ do not agree. 
\end{ex}

\subsection{$GH=\bar{SWIF}$ when $i_{\bdry}(M_j)$ is bounded above: Theorem \ref{thm-injectivity}}\label{sbs-injectivity}

In this subsection we prove Proposition \ref{prop-injectivity}. Theorem \ref{thm-injectivity} follows from it. In Proposition \ref{prop-injectivity} we prove
that the closure of the SWIF limit coincides with the GH limit when having a uniform positive lower bound on the boundary injectivity radii. To prove GH convergence we first prove Lemma \ref{lem-diameters} that gives a uniform bound on the diameters of sequences of inner regions. Then we prove Lemma \ref{lem-bdryInjconv} that shows 
that the sequence of boundaries, $\{(\bdry M_j, d_{M_j})\}$, converges in GH sense. 
Then we apply these lemmas combined with 
our GH convergence theorem for manifolds with boundary, Theorem \ref{thm-GHconvergence1}. To show that $X = \bar Y$ we give an argument that shows that $X_\bdry \subset \bar Y$ and then we apply our GH=SWIF theorem for manifolds with boundary, Theorem \ref{thm-GH=IF}.

First we set some notation. Let $(M,g)$ be a Riemannian manifold with boundary such that the boundary injectivity radius of $M$ satisfies $i_\bdry(M) \geq \iota$, then 
let 
\be\label{eq-normalExpo} 
\gamma : \bdry M \times [0,\iota] \to M
\ee
denote the function that assigns to each $(p, t) \in \bdry M \times [0,\iota]$ the point at time $t$ of the unitary normal geodesic that starts at $p$. Note that $\gamma$ is well defined and a bijection onto its image by the definition of boundary injectivity radius. Finally, let 
\be 
\pi : M \to \bdry M
\ee
the function that assigns to each $p \in M$ a point  $\pi(p) \in \bdry M$ that satisfies $d(p,\pi(p))=d(p, \bdry M)$. Notice that by the boundary injectivity radius bound this point is unique for all $p \in \gamma(\bdry M \times [0,\iota])$.

\begin{lem}\label{lem-diameters}
Let $(M_j,g_j)$ be a sequence of Riemannian manifolds with boundary such that 
\be
i_\bdry(M_j) \geq \iota >\delta_0
\ee
 and 
 \be
 \diam(M^{\delta_0}_j,{d_{M^{\delta_0}_j}})\le D_0
 \ee
for all $j$. Then for all $\delta \in [0, \delta_0]$ and all $j$
the following holds
\be
\diam(M^{\delta}_{j}, {\,d_{M_j^{\delta}}})< D(\delta), 
\ee
where $D(\delta)= D_0 + 2(\delta_0 - \delta)$. 
\end{lem}

\begin{proof}
Let $\delta < \delta_0$. Let's estimate $\diam(M^{\delta}_{j}, {\,d_{M_j^{\delta}}})$. 
Let $p_1, p_2 \in M_j^\delta$. Since $\delta < \delta_0$ and from the definition of inner region we know that  $M_j^{\delta_0} \subset M_j^\delta$. If $p_i \notin M_j^{\delta_0}$, define
\be 
p'_i: =  \gamma (\pi(p_i), \delta_0) \in M_j^{\delta_0}.
\ee
where $\gamma$ is defined in (\ref{eq-normalExpo}). 
This point is well defined since $i_\bdry(M_j)  > \delta_0 > \delta$. Notice that 
\be
d_{M_j}(p_i,p'_i)= \delta_0 - \delta.   
\ee
 If $p_i \in M_j^{\delta_0}$, set $p'_i=p_i$. To end the proof we apply the triangle inequality:
\be
d_{M_j^\delta}(p_1,p_2) \leq d_{M^{\delta_0}}(p'_1,p'_2) + 2(\delta_0 - \delta) \leq D_0 + 2(\delta_0 - \delta).
\ee
\end{proof}

\begin{lem}\label{lem-bdryInjconv}
Let $(M^n_j,g_j)$ be a sequence of Riemannian manifolds with boundary such that 
\be
i_\bdry(M_j) \geq \iota >0.
\ee
 Suppose that there is a decreasing sequence  $\{\delta_i\} \subset \R$ that converges to zero and the inner regions,
$(\overline{M^{\delta_i}_j}, d_{M_j})$, converge in GH sense for all $i$. 

Then a subsequence of $\{(\bdry M_j, d_{M_j})\}$ converges in GH sense. 
\end{lem}

\begin{proof}

By Gromov's compactness theorem, cf. Theorem \ref{thm-GromovCompactness}, it is enough to show that $\{(\bdry M_{j_k}, d_{j_k})\}$ is equibounded and has a uniform diameter bound.

Let $ \delta < \iota$. We claim that if $\{B(x_l,\delta)\}$ is a $\delta$ cover of $(\bdry M_j^\delta, d_j)$ then $\{B(\pi(x_l),3\delta)\}$ is a $3\delta$ cover of $(\bdry M_j, d_j)$.  Let $x \in \bdry M_j$.  Since $\gamma(x,\delta) \in \bdry M_j^\delta$ and $\{B (x_l,\delta)\}$ is a cover of $\bdry M_j^\delta$, there is $l$ such that 
\be
d_j(\gamma(x,\delta),x_l) <\delta.
\ee
 Then,  by the triangle inequality we get
\be 
d_j(x,\pi(x_l))\leq d_j(x,  \gamma(x,\delta))+d_j(\gamma(x,\delta),x_l)+d_j(x_l,\pi(x_l)) <3\delta.
\ee 
This proves the claim.  

Now, since $(\bar{M}^{\delta_i}_j, d_j)$ converges in GH sense, it follows from Gromov's compactness theorem and its converse (cf. Theorem \ref{thm-GromovCompactness} and Theorem \ref{thm-ConverseGromov}), that there is a GH convergent subsequence,  $(\bdry M^{\delta_i}_{j_k}, d_{j_k})$ and there exists a function $N(\,\cdot \, , \{\bdry M^{\delta_i}_{j_k}\})$ as in Definition \ref{defn-equibounded}. 

Without any loss of generality suppose that $\delta_i < \iota$ for all $i$. We define 
\be 
N(3\delta_i): = N(\delta_i,\{\bdry M^{\delta_i}_{j_k}\})
\ee
and extend the domain of $N$ by defining $N(\varepsilon)=N(3\delta_i)$, where $3\delta_i \leq \varepsilon < 3\delta_{i-1}$.   Thus $\{(\bdry M_{j_k}, d_{M_{j_k}})\}$
is equibounded.

Finally, by Lemma \ref{lem-diameters}
\be 
\diam(\bdry M_{j_k}, d_ {\bdry M_{j_k}}) \leq \diam(M_{j_k}) \leq D(0).  
\ee
Thus, by Gromov's Compactness Theorem (cf. Theorem \ref{thm-GromovCompactness}), there is a subsequence of $\{(\bdry M_{j_k}, d_{M_{j_k}})\}$ that converges in GH sense. 
\end{proof}

\begin{prop}\label{prop-injectivity}
Let $n\in \N$ and $\delta,D,V,\theta, \iota >0$, $\iota > \delta$. Suppose that $(M_j,g_j)$ is a sequence of $n$-dimensional compact oriented manifolds with boundary that satisfy 
\be\label{eq-RicVolDiam}
\ric(M_j) \geq 0, \,\, \vol(M_j)\le V,\,\,\diam(M_j^\delta,{d_{M_j^\delta}})\le D,
\ee
\be\label{eq-theta} 
\exists q\in M_j^\delta\,\, \text{such that}\,\,\vol(B(q,\delta)) \ge \theta\delta^n,
\ee
where $B(q,\delta)$ is the ball in $M_j$ with center $q$  and radius $\delta$ and
\be\label{eq-bouInjrad} 
\,i_\bdry(M_j) \geq \iota.
\ee
Then there is a subsequence such that
\be 
(M_{j_k}, d_{j_k}) \GHto (X,d_X) \text{  and  }
(\partial M_{j_k}, d_{j_k}) \GHto (X_\bdry,d_X).
\ee 

If in addition $\vol(\bdry M_j) \leq A$, 
then there is a subsequence and a non zero integral current space such that:
\be
(M_{j_k}, d_{j_k}, T_{j_k}) \Fto (Y \subset X,d_X,T),
\ee
\be
X \setminus X_\bdry \subset Y \,\, \text{ and }\,\, X=\bar Y.
\ee
\end{prop}

\begin{proof}
We only need to show that the hypotheses of Theorem \ref{thm-GH=IF} are satisfied. 
Thus, we first prove the existence of diameter bounds and GH convergence of a subsequence of $(\bdry M_j, d_j)$. 

Choose a decreasing sequence $\{\delta_i\} \subset \R$ that converges to zero such that $\iota > \delta_i$ for all $i$. Using Lemma~\ref{lem-diameters} we obtain the diameter bounds required in Theorem \ref{thm-GH=IF}. Now we need to show that there is a GH convergent subsequence of $(\bdry M_j, d_j)$. By the hypotheses and the diameter bounds that we obtained in Lemma~\ref{lem-diameters} we can apply a result by the author and Sormani \cite{PerSor-2013}, cf. Theorem~\ref{thm-GHdeltas}, to obtain GH convergent subsequences of inner regions:
\be 
(\bar{M}^{\delta_i}_{j_k}, d_{j_k}) \GHto (X(\delta_i),d_{X(\delta_i)}) \,\,\forall i.
\ee 
By Lemma~\ref{lem-bdryInjconv}, there is a further subsequence 
and a metric space $(X_\bdry, d_{X_\bdry})$ such that
\be
(\bdry M_{j_k}, d_{M_{j_k}}) \GHto (X_\partial, d_{\partial}).  
\ee
Then, by Theorem \ref{thm-GH=IF} we have a further GH and SWIF convergent subsequence:
\be \label{just-proven}
(M_{j_k}, d_{j_k}) \GHto (X,d_X)
\ee
and
\be \label{just-proven}
(M_{j_k}, d_{j_k},T_{j_k} ) \Fto (Y \subset X,d_X,T)
\ee
such that $X\setminus X_\bdry \subset Y$.

To prove that $X=\bar Y$ it remains to prove that $X_\partial \subset \bar Y$. With no loss of generality we suppose that $(M_j, d_j)$ converges in GH sense. Let $x \in X_\partial$ and $x_j \in \partial M_j$ be a sequence that converges to $x$. 
For all $i$, by the GH convergence of $(M_j,d_j)$, there is a subsequence $j_k$ such that 
\be \label{eq-inY}
\gamma_{j_k}(x_{j_k},\delta_i)  \to y(\delta_i)\,\,\text{as}\,\,k\to 0,
\ee 
where $\gamma_{j_k}$ denotes the normal exponential function defined on $\bdry M_{j_k} \times [0,\iota]$, see (\ref{eq-normalExpo}). From (\ref{eq-inY}) and the GH convergence of $(M_j^{\delta_i}, d_j)$ we know that $y(\delta_i) \subset X(\delta_i) \subset Y$.

Using the triangle inequality we get $d_X(y(\delta_i),x)=\delta_i$. Hence, 
\be 
y(\delta_i) \in Y \to x \,\,\text{as}\,\,i\to 0.
\ee 
This proves that $x \in \bar Y$. Thus, $X_\partial\subset \bar{Y}$. This finishes the proof. 
\end{proof}

\begin{proof}[Proof of Theorem \ref{thm-injectivity}]
From Proposition \ref{prop-injectivity} we see that there is a subsequence, a compact metric space $(X, d_X)$ and an $n$-dimensional integral current space, $(Y \subset X,d_X,T)$
such that
\be
(M_{j_k}, d_{j_k}) \GHto (X,d_X),
\ee
\be \label{just-proven}
(M_{j_k}, d_{j_k},T_{j_k} ) \Fto (Y,d_X,T)
\ee
and $X = \bar Y$. $Y$ is $\mathcal H^n$ countably rectifiable by definition of $n$-dimensional integral current space. 
\end{proof}

\begin{rmrk}
If all the hypotheses of Theorem \ref{thm-injectivity} hold except 
the uniform positive lower bound on the boundary injectivity radii then the conclusion of the theorem might not hold. See Example \ref{ex-oneSpline} where $\lim_{j\to \infty} i_\bdry(M_j) = 0$ due to the increasingly thin splines. 
\end{rmrk}


The next example shows that although the injectivity radius of the boundary is a popular assumption in theorems about convergence of
manifolds with boundary, it is not a necessary condition. 

\begin{ex}\label{ex-injFail}
Let  $M_j=\overline{B(0,r)} \cup A_j$ be the sequence in Euclidean space that consists of a closed ball of radius $r$ with an increasingly thin and short spline $A_j$ attached to the ball such that $M_j$ converges in GH and SWIF sense to $\overline{B(0,r)}$. See Example \ref{ex-oneSpline} and Figure \ref{fig-OneSpline}. There the splines have constant length.  
\end{ex}

\section{GH Convergence of $(\bdry M_j, d_{M_j})$}\label{sec-ConvergenceBoun}

Let $(M,g)$ be a Riemannian manifold with smooth boundary. We denote by $d_M$ the metric given by $g$. Since $M$ has smooth boundary, $\bdry M$ can be endowed with two different metrics, $d_M$ which is the restriction of $d_M$ to $\bdry M$ and $d_{\bdry M}$ which is the metric given by the Riemannian metric of $\bdry M$.  

Some of our GH compactness theorems require GH convergence of the sequence $(\bdry M_j, d_{M_j})$.  Observe that GH convergence of $(\bdry M_j, d_{\bdry M_j})$ implies GH convergence of $(\bdry M_j, d_{M_j})$ provided each $(\bdry M_j, d_{\bdry M_j})$ is connected or have a bounded number of connected components (cf. Proposition~\ref{prop-epsCovers}). Thus, by uniformly bounding the Ricci curvature of $\bdry M_j$ we will prove a GH compactness theorem,   
Theorem \ref{thm-bdryConv}, for $(\bdry M_j, d_{\bdry M_j})$ and $(\bdry M_j, d_{\bdry M_j})$.

\begin{thm}\label{thm-bdryConv}
Let $\{(M^n_j, g_j)\}$ be a sequence of Riemannian manifolds with smooth boundary
such that $\ric(M_j)\geq 0$, 
\be\label{eq-bdryConv} 
\diam(\bdry M_j, d_{\bdry M_j})\leq D_\bdry,\,\,  (\cur(e_n,X)X,e_n) \leq \gamma \,\,\text{and}\,\,\, \alpha \leq B(X,X) \leq \beta,
\ee
where $e_n$ denotes the normal unitary vector field, $B$ is the second fundamental form of $\partial M$ and $X$ is a vector field  in $T \bdry M$ such that $\nabla_{e_n}X=0$.
Then, there is a subsequence $\{j_k\}$ such that 
both $(\bdry M_{j_k}, d_{\bdry M_{j_k}})$ and $(\bdry M_{j_k}, d_{j_k})$ converge in GH sense. 
\end{thm}

We now prove two propositions which we will apply to prove this theorem.

\begin{prop}\label{prop-epsCovers}
Let $(M,g)$ be a Riemannian manifold with boundary.  If $\{B_{d_{\bdry M}}(x_i,\varepsilon)\}$ is a cover of $(\bdry M, d_{\bdry M})$ then $\{B_{d_{M}}(x_i,\varepsilon)\}$ is a cover of $(\bdry M, d_{M})$.
\end{prop}

\begin{proof}
It is enough to show that for all $x \in \bdry M$ the following holds
\be 
B_{d_{\bdry M}}(x,\varepsilon) \subset B_{d_M}(x,\varepsilon) .
\ee
By the definition of $d_M$ and $d_{\bdry M}$ we know that 
\be
d_M(x,x') \leq d_{\bdry M}(x,x') \textrm{ for all }x,x' \in \bdry M.
\ee
 Thus, 
 \be
 d_{\bdry M}(x,x') < \varepsilon \textrm{ implies }d_M(x,x') < \varepsilon.
 \ee
  Hence, 
  \be
  B_{d_{\bdry M}}(x,\varepsilon)\subset B_{d_M}(x,\varepsilon).
  \ee
\end{proof}

\begin{prop}\label{prop-RiciBounded}
Let $(M,g)$ be a Riemannian manifold with boundary. If 
\be 
\ric(M)\geq 0,\,\,(\cur(e_n,X)X,e_n) \leq \gamma \,\,\text{and}\,\,\, \alpha \leq B(X,X) \leq \beta.
\ee
Then, 
\be
\ric_\bdry(X,X) \geq c(n,\gamma,\alpha,\beta)=-\gamma +(n-1)(\alpha^2-\beta^2).
\ee
\end{prop}

\begin{proof}
Let $p \in \bdry M$. Choose an orthonormal basis $e_i$ of $T_{p} M$ such that $e_n$ is perpendicular to $T_p \bdry M$ and $\nabla_{e_n}e_i=0$ for $1 \leq i \leq n-1$. Using  Gauss formula we get
\be
(\cur_\bdry (e_i,e_j)e_j,e_i) = (\cur(e_i,e_j)e_j,e_i) + B(e_j,e_j)B(e_i,e_i)-B^2(e_j,e_i),
\ee
for $1 \leq i,j \leq n-1$. Adding over $i$ and adding and substracting $(\cur(e_n,e_j)e_j,e_n)$ we obtain
\begin{eqnarray}
\ric_\bdry(e_j,e_j) &=& \ric(e_j,e_j)\, - \,(\cur(e_n,e_j)e_j,e_n)\\
&&\quad +\,\,\, B(e_j,e_j)\,\,\sum_{i=1}^{n-1} B(e_i,e_i)
\,\,-\,\,\sum_{i=1}^{n-1} B^2(e_i,e_j). 
\end{eqnarray}
\end{proof}

\begin{proof}[Proof of Theorem \ref{thm-bdryConv}]
We know that 
\be
\diam(\bdry M_j, d_{\bdry M_j})\leq D_\bdry.
\ee
From Proposition \ref{prop-RiciBounded} we get 
\be
\ric_\bdry(X,X) \geq -\gamma +(n-1)(\alpha^2-\beta^2).
\ee
 Thus, by Gromov's Ricci Compactness Theorem (cf. Theorem~\ref{thm-GromovRic}) there is a GH convergent subsequence $(\bdry M_{j_k}, d_{\bdry M_{j_k}})$ and this subsequence is equibounded (cf. Theorem~\ref{thm-ConverseGromov}). Then 
$(\bdry M_{j_k}, d_{j_k})$ is equibounded by Proposition~\ref{prop-epsCovers}. 
Moreover, by the definition of the restricted metric and the induced length metric,  
\be
\diam(\bdry M_{j_k}, d_{j_k}) \leq \diam(\bdry M_{j_k}, d_{\bdry M_{j_k}})\leq D_\bdry.
\ee
Then by Gromov's Compactness theorem there is a subsequence of $(\bdry M_{j_k}, d_{j_k})$ that converges in GH sense. 
\end{proof}


\bibliographystyle{plain}
\bibliography{bib2013}

\begin{thebibliography}{10}

\bibitem{AK}
Luigi Ambrosio and Bernd Kirchheim.
\newblock Currents in metric spaces.
\newblock {\em Acta Math.}, 185(1):1--80, 2000.

\bibitem{Anderson-2004}
Michael Anderson, Atsushi Katsuda, Yaroslav Kurylev, Matti Lassas, and Michael
  Taylor.
\newblock Boundary regularity for the {R}icci equation, geometric convergence,
  and {G}el\cprime fand's inverse boundary problem.
\newblock {\em Invent. Math.}, 158(2):261--321, 2004.

\bibitem{BBI}
Dmitri Burago, Yuri Burago, and Sergei Ivanov.
\newblock {\em A course in metric geometry}, volume~33 of {\em Graduate Studies
  in Mathematics}.
\newblock American Mathematical Society, Providence, RI, 2001.

\bibitem{ChCo-PartI}
Jeff Cheeger and Tobias~H. Colding.
\newblock On the structure of spaces with {R}icci curvature bounded below. {I}.
\newblock {\em J. Differential Geom.}, 46(3):406--480, 1997.

\bibitem{Colding-volume}
Tobias~H. Colding.
\newblock Ricci curvature and volume convergence.
\newblock {\em Ann. of Math. (2)}, 145(3):477--501, 1997.

\bibitem{Colding-survey}
Tobias~H. Colding.
\newblock Spaces with {R}icci curvature bounds.
\newblock In {\em Proceedings of the {I}nternational {C}ongress of
  {M}athematicians, {V}ol. {II} ({B}erlin, 1998)}, number Extra Vol. II, pages
  299--308 (electronic), 1998.

\bibitem{FF}
Herbert Federer and Wendell~H. Fleming.
\newblock Normal and integral currents.
\newblock {\em Ann. of Math. (2)}, 72:458--520, 1960.

\bibitem{Greene-Petersen}
Robert~E. Greene and Peter Petersen~V.
\newblock Little topology, big volume.
\newblock {\em Duke Math. J.}, 67(2):273--290, 1992.

\bibitem{Gromov-groups}
M.~Gromov.
\newblock Hyperbolic groups.
\newblock In {\em Essays in group theory}, volume~8 of {\em Math. Sci. Res.
  Inst. Publ.}, pages 75--263. Springer, New York, 1987.

\bibitem{Gromov-81a}
Mikhael Gromov.
\newblock Groups of polynomial growth and expanding maps.
\newblock {\em Inst. Hautes \'Etudes Sci. Publ. Math.}, 53:53--73, 1981.

\bibitem{Gromov-filling}
Mikhael Gromov.
\newblock Filling {R}iemannian manifolds.
\newblock {\em J. Differential Geom.}, 18(1):1--147, 1983.

\bibitem{Gromov-metric}
Misha Gromov.
\newblock {\em Metric structures for {R}iemannian and non-{R}iemannian spaces},
  volume 152 of {\em Progress in Mathematics}.
\newblock Birkh\"auser Boston Inc., Boston, MA, 1999.
\newblock Based on the 1981 French original [ MR0682063 (85e:53051)], With
  appendices by M. Katz, P. Pansu and S. Semmes, Translated from the French by
  Sean Michael Bates.

\bibitem{Knox-2012}
Kenneth Knox.
\newblock A compactness theorem for riemannian manifolds with boundary and
  applications.
\newblock {\em arXiv:1211.6210 [math.DG]}, pages 1--17, 2012.

\bibitem{Kodani-1990}
Shigeru Kodani.
\newblock Convergence theorem for {R}iemannian manifolds with boundary.
\newblock {\em Compositio Math.}, 75(2):171--192, 1990.

\bibitem{Li-Perales}
Nan Li and Raquel Perales.
\newblock On the {S}ormani-{W}enger intrinsic flat convergence of {A}lexandrov
  spaces.
\newblock {\em arXiv:1411.6854v1 [math.MG]}, 2014.

\bibitem{Morgan-text}
Frank Morgan.
\newblock {\em Geometric measure theory}.
\newblock Elsevier/Academic Press, Amsterdam, fourth edition, 2009.
\newblock A beginner's guide.

\bibitem{Perales-Vol}
Raquel Perales.
\newblock Volumes and limits of manifolds with ricci curvature and mean
  curvature bounds.
\newblock {\em arXiv:1404.0560v3 [math.MG]}, pages 1--9, 2014.

\bibitem{PerSor-2013}
Raquel Perales and Christina Sormani.
\newblock Sequences of open {R}iemannian manifolds with boundary.
\newblock {\em Pacific J. Math.}, 270(2):423--471, 2014.

\bibitem{Perelman-max-vol}
G.~Perelman.
\newblock Manifolds of positive {R}icci curvature with almost maximal volume.
\newblock {\em J. Amer. Math. Soc.}, 7(2):299--305, 1994.

\bibitem{PorSor}
Jacobus Portegies and Christina Sormani.
\newblock Properties of the {I}ntrinsic {F}lat {D}istance.
\newblock {\em arXiv:1210.3895v3 [math.DG]}, 2014.

\bibitem{SorAA}
Christina Sormani.
\newblock Intrinsic flat {A}rzela-{A}scoli theorems.
\newblock {\em arXiv:1402.6066 [math.MG]}, pages 1--33, 2014.

\bibitem{SorWen1}
Christina Sormani and Stefan Wenger.
\newblock Weak convergence and cancellation, appendix by {R}aanan {S}chul and
  {S}tefan {W}enger.
\newblock {\em Calculus of Variations and Partial Differential Equations},
  38(1-2), 2010.

\bibitem{SorWen2}
Christina Sormani and Stefan Wenger.
\newblock The intrinsic flat distance between riemannian manifolds and other
  integral current spaces.
\newblock {\em Journal of Differential Geometry}, 87:117--199, 2011.

\bibitem{Wenger-flat}
Stefan Wenger.
\newblock Flat convergence for integral currents in metric spaces.
\newblock {\em Calc. Var. Partial Differential Equations}, 28(2):139--160,
  2007.

\bibitem{Wenger-compactness}
Stefan Wenger.
\newblock Compactness for manifolds and integral currents with bounded diameter
  and volume.
\newblock {\em Calc. Var. Partial Differential Equations}, 40(3-4):423--448,
  2011.

\bibitem{Wong-2008}
Jeremy Wong.
\newblock An extension procedure for manifolds with boundary.
\newblock {\em Pacific J. Math.}, 235(1):173--199, 2008.

\end{thebibliography}

\end{document}